\documentclass[a4paper,11pt]{article}
\usepackage[latin1]{inputenc}
\usepackage[english]{babel}
\usepackage{amsmath}
\usepackage{amsfonts}
\usepackage{amssymb}
\usepackage{epsfig}
\usepackage{amsopn}
\usepackage{amsthm}
\usepackage{color}
\usepackage{graphicx}
\usepackage{subfigure}
\usepackage{enumerate}
\newtheorem{theorem}{Theorem}[section]

\newtheorem{lemma}[theorem]{Lemma}
\newtheorem{proposition}[theorem]{Proposition}
\newtheorem{definition}[theorem]{Definition}

\newtheorem*{theorem*}{Theorem}
\newtheorem*{lemma*}{Lemma}
\newtheorem*{remark*}{Remark}
\newtheorem*{definition*}{Definition}
\newtheorem*{proposition*}{Proposition}
\newtheorem*{corollary*}{Corollary}
\numberwithin{equation}{section}
\newcommand{\real}{\mathbb{R}}

\let\ced=\c         

\def\qed{\,\unskip\kern 6pt \penalty 500
\raise -2pt\hbox{\vrule \vbox to8pt{\hrule width 6pt
\vfill\hrule}\vrule}\par}
\definecolor{darkblue}{rgb}{0.05, .05, .65}
\definecolor{darkgreen}{rgb}{0.1, .65, .1}
\definecolor{darkred}{rgb}{0.8,0,0}
\newcommand{\beqn}{\begin{equation}}
\newcommand{\eeqn}{\end{equation}}
\newcommand{\bear}{\begin{eqnarray}}
\newcommand{\eear}{\end{eqnarray}}
\newcommand{\bean}{\begin{eqnarray*}}
\newcommand{\eean}{\end{eqnarray*}}
\begin{document}

\title{\huge \bf Eternal solutions for a reaction-diffusion equation with weighted reaction}

\author{
\Large Razvan Gabriel Iagar\,\footnote{Departamento de Matem\'{a}tica
Aplicada, Ciencia e Ingenieria de los Materiales y Tecnologia
Electr\'onica, Universidad Rey Juan Carlos, M\'{o}stoles,
28933, Madrid, Spain, \textit{e-mail:} razvan.iagar@urjc.es},
\\[4pt] \Large Ariel S\'{a}nchez,\footnote{Departamento de Matem\'{a}tica
Aplicada, Ciencia e Ingenieria de los Materiales y Tecnologia
Electr\'onica, Universidad Rey Juan Carlos, M\'{o}stoles,
28933, Madrid, Spain, \textit{e-mail:} ariel.sanchez@urjc.es}\\
[4pt] }
\date{}
\maketitle

\begin{abstract}
We prove existence and uniqueness of \emph{eternal solutions} in self-similar form growing up in time with exponential rate for the weighted reaction-diffusion equation
$$
\partial_tu=\Delta u^m+|x|^{\sigma}u^p,
$$
posed in $\real^N$, with $m>1$, $0<p<1$ and the critical value for the weight
$$
\sigma=\frac{2(1-p)}{m-1}.
$$
Existence and uniqueness of some specific solution holds true when $m+p\geq2$. On the contrary, no eternal solution exists if $m+p<2$. We also classify exponential self-similar solutions with a different interface behavior when $m+p>2$. Some transformations to reaction-convection-diffusion equations and traveling wave solutions are also introduced.
\end{abstract}

\


\smallskip


\section{Introduction}

The goal of this paper is to investigate the availability of \emph{eternal solutions} of exponential self-similar type for the reaction-diffusion equation with weighted reaction
\begin{equation}\label{eq1}
u_t=\Delta u^m+|x|^{\sigma}u^p,
\end{equation}
posed for $(x,t)\in\real^N\times(0,\infty)$, in the following range of exponents
\begin{equation}\label{exp}
m>1, \qquad 0<p<1, \qquad \sigma=\frac{2(1-p)}{m-1}.
\end{equation}
More precisely, we look for radially symmetric solutions of the form
\begin{equation}\label{SSS}
u(x,t)=e^{\alpha t}f(\xi), \qquad \xi=|x|e^{-\beta t},
\end{equation}
with self-similar exponents $\alpha$, $\beta$ to be determined. Let us notice at this point that solutions in the form \eqref{SSS} are in fact solutions that can be defined for any $t\in\real$ (that is, also backward in time) and this is why they are usually called in literature \emph{eternal solutions}. The availability of \emph{eternal solutions} is a very uncommon feature for parabolic equations, due to the fact that they usually enjoy smoothing effects and this implies a specific sense of the time flow, as it is for example well-known for the heat equation or the porous medium equation. That is why, only some very specific parabolic equations allow the existence of such solutions. To give some noticeable examples of eternal solutions from previous literature, in the case of the fast diffusion equation they were established for the critical case $m=m_c$, where $m_c=(N-2)_{+}/N$ \cite{GPV00, VazSm}, and then also discovered for the parabolic $p$-Laplacian equation with the specific exponent $p_c=2N/(N+1)$ through a change of variable between the two equations \cite{ISV08}. Another significant model allowing for eternal solutions is the logarithmic diffusion posed in dimension $N=2$, which is related to the Ricci flow in $\real^2$ according to Daskalopoulos and Sesum \cite{DS06}, while such an interface between two regimes with different behaviors takes place for the porous medium equation with standard absorption
$$
u_t=\Delta u^m-u^q, \qquad m>1, \ q>0
$$
for the precise value $q=1$, as it is given by the investigation of the limiting ranges in \cite{VazquezSurvey}. More recently, \emph{eternal solutions} were established and completely classified by Lauren\ced{c}ot and one of the authors in \cite{IL13} for a parabolic equation combining $p$-Laplacian fast diffusion and gradient-type absorption, namely
$$
u_t={\rm div}(|\nabla u|^{p-2}\nabla u)-|\nabla u|^{p/2}, \qquad {\rm with} \ p_c=\frac{2N}{N+1}<p<2.
$$
In all these cases, the existence of eternal solutions is linked to specific exponents that realize the interface between two regimes of different behaviors, which in all the above mentioned cases are the range of algebraic (power-like) time decay as $t\to\infty$ (which occurs for higher exponents than the specific one where eternal solutions exist) and the range of finite time extinction (which occurs for smaller exponents). Moreover, in the previously mentioned cases the classification of exponential (eternal) solutions is usually very complicated. Just as an example, for the critical case of the porous medium equation
$$
u_t=\Delta u^m, \qquad {\rm with} \ m=m_c=\frac{N-2}{N}, \ N\geq3,
$$
it is easy to obtain eternal solutions decaying exponentially in time (see \cite[Section 5.6.1]{VazSm}) but solutions with a super-exponential decay, namely $\exp(-Ct^{N/(N-2)})$ where $C>0$, were also constructed in \cite{GPV00}. Thus, the dynamic of some of these specific cases in which eternal solutions do exist may be rather complex and it is always an interesting problem to study.

Coming back to our Eq. \eqref{eq1}, there are two facts indicating that \emph{eternal solutions}, in this case with exponential grow-up in time, might exist. On the one hand, we have shown in two recent papers \cite{IS20b, IS21b} that self-similar blow-up profiles to Eq. \eqref{eq1} with $m>1$, $p\in(0,1)$ (only in dimension $N=1$) exist precisely for $\sigma>2(1-p)/(m-1)$, and we also classified them. This suggests us that the precise value of $\sigma$ in \eqref{exp} might be a borderline case in the style of the above quoted ones, limiting in this case the regime of finite time blow-up and (as it is not yet proved but expected) the regime of algebraic grow-up in time. On the other hand, an easy change of variable between the following two equations
$$
u_t=\Delta u^m+u, \qquad {\rm and} \qquad v_t=\Delta v^m,
$$
the former being exactly the limit case of our range \eqref{exp} with $\sigma=0$, $p=1$, gives a rather big number of eternal solutions for the first equation, corresponding to all the self-similar equations (with algebraic decay) for the second one. This is how we decided to address the problem of studying the availability of \emph{eternal solutions} to Eq. \eqref{eq1} for the range of exponents \eqref{exp} and classify them if possible. By introducing the ansatz \eqref{SSS} into Eq. \eqref{eq1} we readily obtain that the self-similar exponents should satisfy
\begin{equation}\label{relation}
\alpha=\frac{2}{m-1}\beta,
\end{equation}
while the profile $f(\xi)$ solves the following differential equation
\begin{equation}\label{ODE}
(f^m)''(\xi)+\frac{N-1}{\xi}(f^m)'(\xi)-\alpha f(\xi)+\beta\xi f'(\xi)+\xi^{\sigma}f(\xi)^p=0.
\end{equation}
Let us notice here that the single relation \eqref{relation} does not ensure the positivity of $\alpha$ and $\beta$. But this is easy to see from the mass calculation. Defining
$$
M(t)=\int_{\real^N}u(x,t)\,dx,
$$
it readily follows from Eq. \eqref{eq1} that the mass is increasing in time due to the reaction term. But if we plug the ansatz \eqref{SSS} into the definition of the mass, we find that for eternal exponential self-similar solutions the mass evolves as it is indicated below
$$
M(t)=e^{(\alpha+N\beta)t}\int_{\real^N}f(y)\,dy,
$$
which forces $\alpha+N\beta>0$. This, together with \eqref{relation}, give the positivity of both exponents.

We will thus look throughout this paper to profiles $f(\xi)$ solutions to Eq. \eqref{ODE}. Let us mention here that we are only interested in non-negative profiles, that is $f(\xi)\geq0$ for any $\xi\geq0$. More precisely, we introduce the following
\begin{definition}
A solution $f(\xi)$ to Eq. \eqref{ODE} is called a \textbf{good profile with interface} if it satisfies simultaneously the following conditions:

(a) $f(0)=A>0$, $f'(0)=0$.

(b) There exists $\xi_0\in(0,\infty)$ such that $f(\xi_0)=0$, $f(\xi)>0$ for any $\xi\in(0,\xi_0)$ and the interface condition $(f^m)'(\xi_0)=0$ is fulfilled.
\end{definition}
We stress here that in our previous papers \cite{IS20b, IS21b} there were different types of good profiles with interface, in particular profiles starting with $f(0)=0$ and satisfying a left-interface condition at $\xi=0$, namely $(f^m)'(0)=0$. The subsequent analysis will show that such profiles do no longer exist in our range of exponents \eqref{exp}, this is why we do not consider them in the definition. Moreover, in other papers investigating eternal solutions, strictly positive profiles presenting a tail as $\xi\to\infty$ are obtained \cite{GPV00, IL13}. In our case we also do not have such profiles, as we shall prove that all the solutions (of any type) to Eq. \eqref{ODE} have to intersect the axis $\xi=0$ at some finite point. Another significant difference between the case we are dealing with here and the ranges of exponents studied in our previous works \cite{IS20b, IS21b} is the fact that \emph{the self-similar exponents $\alpha$ and $\beta$ are not fixed}: one of them is totally free and is in fact the most important parameter of the problem. This is a situation which is characteristic to the very specific cases where eternal solutions exist, see for example \cite{IL13}.

It seems that there are two free parameters in the problem: the value $f(0)=A$ and one of the self-similar exponents, let us choose $\alpha>0$. But it is straightforward to check that if $f(\xi)$ is a good profile with interface such that $f(0)=A>0$, the following rescaling
\begin{equation}\label{resc}
g_{\lambda}(\xi)=\lambda^{-2/(m-1)}f(\lambda\xi), \qquad g_{\lambda}(0)=\lambda^{-2/(m-1)}A
\end{equation}
gives another good profile to Eq. \eqref{ODE}. It is thus sufficient to fix $A=1$ and study the initial value problem for Eq. \eqref{ODE} with $f(0)=1$, $f'(0)=0$. We will thus restrict in the statement of the main results of this work to profiles with $f(0)=1$. We are now in a position to state the main results of this paper.
\begin{theorem}[Existence and uniqueness when $m+p>2$]\label{th.1}
Assume that $m$, $p$ and $\sigma$ are as in \eqref{exp} and that $m+p>2$. Then there exists a unique exponent $\alpha^*>0$ such that good profiles with interface such that $f(0)=1$ exist for any self-similar exponent $\alpha\in[\alpha^*,+\infty)$ and do not exist for $\alpha\in(0,\alpha^*)$. Moreover, we can classify them according to their interface behavior as follows:

\medskip

(a) For $\alpha=\alpha^*$, there exists a \textbf{unique} good profile with interface $f(\xi)$ having the interface behavior
\begin{equation}\label{TypeI}
f(\xi)\sim C(\xi_0-\xi)_{+}^{1/(m-1)}, \qquad {\rm as} \ \xi\to\xi_0 \ \ ({\rm interface \ of \ Type \ I})
\end{equation}

(b) For any $\alpha\in(\alpha^*,\infty)$ fixed, there exists a \textbf{unique} good profile with interface $f(\xi)$ having the interface behavior
\begin{equation}\label{TypeII}
f(\xi)\sim C(\xi_0-\xi)_{+}^{1/(1-p)}, \qquad {\rm as} \ \xi\to\xi_0 \ \ ({\rm interface \ of \ Type \ II})
\end{equation}
\end{theorem}
We notice that there are two different interface behaviors, which are totally different for $m+p>2$. These different interface behaviors have been noticed first for the reaction-diffusion equation without weight, that is with $\sigma=0$ (see for example \cite{dP94}) and more recently these two different interfaces were analyzed in our recent paper \cite{IS20b} in connection with the \emph{interface equation}, showing that they are essentially different since they satisfy two completely different differential equations for their speed of advance. The presence of these two types of interface is a characteristic of the range $0<p<1$. 

\medskip 

\noindent \textbf{Remark.} The classification in Theorem \ref{th.1} reminds us about the classification of the traveling waves to, for example, the celebrated Fisher-KPP equation and more general reaction-diffusion-convection equations \cite{GiBook}, for which there exists a unique minimal speed $c^*$ for which traveling wave solutions exist if and only if $c\geq c^*$ and the traveling wave with $c=c^*$ is unique. It is not completely surprising the fact that we can map our equation into a reaction-convection-diffusion equation (see Section \ref{sec.transf}) and then use our results to classify its traveling wave solutions.

\medskip

When $m+p=2$ the two types of interfaces coincide from a qualitative point of view, but we can obtain further information by getting an explicit value of $\alpha^*$.
\begin{theorem}[Existence and uniqueness when $m+p=2$]\label{th.1bis}
Assume that $m$, $p$ and $\sigma$ are as in \eqref{exp} and that $m+p=2$. Let
\begin{equation}\label{alpha.equal}
\alpha^*=\frac{4\sqrt{m}}{m-1}.
\end{equation}
Then, for any $\alpha\in[\alpha^*,\infty)$ fixed, there exists a \textbf{unique} good profile with interface $f(\xi)$ with $f(0)=1$ and having the interface behavior given by \eqref{TypeI}. There is no good profile with interface for $\alpha\in(0,\alpha^*)$.
\end{theorem}
In the complementary case $m+p<2$ the result is very simple.
\begin{theorem}[Non-existence when $m+p<2$]\label{th.2}
If $m$, $p$ and $\sigma$ are as in \eqref{exp} and $m+p<2$, there exists no good profile with interface.
\end{theorem}
We stress here that the sign of the expression $m+p-2$ produces a big difference in the analysis of both the ordinary differential equation Eq. \eqref{ODE} and of Eq. \eqref{eq1} from a qualitative point of view. The influence of the sign of $m+p-2$ in the behavior of reaction-diffusion equations with $p\in(0,1)$ has been observed for the first time by de Pablo and V\'azquez in \cite{dPV90} for the non-weighted reaction-diffusion equation in dimension $N=1$ and the explanation for such a huge influence on the properties of the equation is that, on the one hand, if $m+p\geq2$ the equation presents \emph{finite speed of propagation} of supports, while on the other hand for $m+p<2$ \emph{infinite speed of propagation} holds true. The latter means that, even if the initial condition $u_0$ is compactly supported, any solution becomes strictly positive at any $t>0$. This implies that interfaces cannot exist if $m+p<2$. The same dichotomy has been noticed for $\sigma>2(1-p)/(m-1)$ in \cite{IS20b}.

\medskip

\noindent \textbf{Organization of the paper}. The technique used for the proofs is a phase-plane analysis, related to an autonomous dynamical system associated to Eq. \eqref{ODE}. The dynamical system that we use in the proofs is very different from the ones we used (in three dimensions) in our recent papers \cite{IS20b, IS21b}, since that one is no longer satisfactory at a technical level. Thus, we will go back to the variables employed for the dynamical system in the older paper \cite{ISV08}. As we shall see, there are some technical differences between the cases $m+p>2$ and $m+p=2$ and also between the study of the dynamical system in dimension $N\geq3$ and, respectively, in dimensions $N=2$ and $N=1$. We will thus do our main study of the dynamical system taking as core of the analysis the range of exponents when $m+p>2$ and $N\geq3$. This analysis, leading to the proof of Theorem \ref{th.1} for this range of exponents and dimension, will be performed in the longest Section \ref{sec.large}, which is conveniently sub-divided into several subsections. It then follows a Section \ref{sec.equal} devoted to the case $m+p=2$ also in dimension $N\geq3$, emphasizing only on the differences with respect to the previous section and leading to the proof of Theorem \ref{th.1bis}. The specific dimensions $N=2$ and $N=1$, always with $m+p\geq2$, will be studied in a shorter Section \ref{sec.lowdim} where only the differences with respect to the previous proofs for $N\geq3$ will be given. The proof of Theorem \ref{th.2} will readily follow from a similar phase-plane analysis, by noticing that there is no possible interface behavior. This proof will be addressed at the same time for all dimensions in Section \ref{sec.lower}. The final Section \ref{sec.transf} will be dedicated to a transformation mapping Eq. \eqref{eq1} into a reaction-convection-diffusion equation in dimension $N=1$, and an application of our previous analysis leads to a \emph{classification of the traveling wave solutions} to this new equation. In particular, the uniqueness of the traveling wave of minimal speed is proved.

\section{Analysis of the range $m+p>2$ in dimension $N\geq3$}\label{sec.large}

This section is devoted to the main part of the proof of Theorem \ref{th.1}, that is, for the range of exponents when $m+p>2$ and in dimension $N\geq3$. As we shall see in Section \ref{sec.lowdim}, dimensions $N=2$ and $N=1$ introduce some technical differences despite the fact that the qualitative results are the same. This is why, we will present the main arguments of the proofs avoiding such technical exceptions and we deal with them lately. The idea of the proof is to transform the non-autonomous partial differential equation \eqref{ODE} into an autonomous dynamical system and then study the associated phase plane. In our recent papers devoted to the similar range of exponents $m>1$ and $p\in(0,1)$ we used a three-dimensional dynamical system, but in this precise case we can reduce it to a two-dimensional one.

\subsection{The dynamical system. Local analysis}\label{subsec.finite}

Retaking a change of variables used in \cite{ISV08}, we introduce the new variables
\begin{equation}\label{change}
X=\frac{\alpha}{2m}\xi^2f(\xi)^{1-m}, \qquad Y=\xi f'(\xi)f^{-1}(\xi), \qquad \frac{d}{d\eta}=m^2\xi\frac{d}{d\xi},
\end{equation}
to obtain after straightforward calculations the following autonomous system of differential equations
\begin{equation}\label{PPSyst}
\left\{\begin{array}{ll}\dot{X}=X(2-(m-1)Y),\\ \dot{Y}=-mY^2-(N-2)Y+2X-(m-1)XY-KX^{(m-p)/(m-1)},\end{array}\right.
\end{equation}
depending on the parameter
\begin{equation}\label{param}
K=\frac{1}{m}\left(\frac{2m}{\alpha}\right)^{(m-p)/(m-1)}.
\end{equation}
The dot notation in the system \eqref{PPSyst} means derivative with respect to the independent variable $\eta$ introduced in \eqref{change}. Let us notice here that $X\geq0$, $Y$ can change sign and that in our range of parameters,
$$
1<\frac{m-p}{m-1}<2, \qquad {\rm for \ any} \ m>1, \ p\in(0,1), \ {\rm such \ that} \ m+p>2,
$$
a fact of utmost importance at technical level in the subsequent analysis. We readily find that the system \eqref{PPSyst} has two critical points in the finite part of the phase plane, namely
$$
P_0=(0,0), \qquad P_1=\left(0,-\frac{N-2}{m}\right),
$$
whose local analysis is performed in the next two Lemmas below.
\begin{lemma}[Local analysis near $P_0$]\label{lem.P0}
The critical point $P_0$ is a saddle point. There is a unique orbit going out of it, which contains profiles such that $f(0)=A>0$, $f'(0)=0$.
\end{lemma}
\begin{proof}
The linearization of the system \eqref{PPSyst} near $P_0$ has the matrix
$$M(P_0)=\left(
  \begin{array}{cc}
    2 & 0 \\
    2 & 2-N \\
  \end{array}
\right),$$
thus $P_0$ is a saddle point with eigenvalues $\lambda_1=2$, $\lambda_2=2-N<0$, recalling that we work in dimension $N\geq3$. We are interested in the unique orbit going out of it, which is tangent to the eigenvector $e_1=(1,2/N)$ corresponding to the eigenvalue $\lambda_1$ of the matrix $M(P_0)$. With respect to the local behavior, the profiles contained in this orbit satisfy $X/Y\sim N/2$, which by substitution in terms of profiles and direct integration leads to
\begin{equation}\label{interm1}
f(\xi)\sim\left(D+\frac{\alpha(m-1)}{2mN}\xi^2\right)^{1/(m-1)}, \qquad D>0 \ {\rm free \ constant}.
\end{equation}
The previous local approximation may only take place as $\xi\to0$. Indeed, if \eqref{interm1} would take place as $\xi\to\xi_0\in(0,\infty)$ or as $\xi\to\infty$, recalling that $X\to0$ near $P_0$ we infer that $f^{1-m}(\xi)\sim0$, which contradicts \eqref{interm1} if $\xi\to\xi_0\in(0,\infty)$. On the other hand, the possibility that $\xi\to\infty$ is ruled out by the fact that $Y\to0$. Indeed, if \eqref{interm1} holds true as $\xi\to\infty$ we find that $Y(\xi)\sim 1/(m-1)$ as $\xi\to\infty$ and a contradiction. Thus, \eqref{interm1} can only hold true as $\xi\to0$, and we readily get that $f'(0)=0$ and $f(0)=D^{1/(m-1)}>0$.
\end{proof}
\begin{lemma}[Local analysis near $P_1$]\label{lem.P1}
The critical point $P_1$ is an unstable node. The orbits going out of it contain profiles such that present a vertical asymptote at $\xi=0$, more precisely
\begin{equation}\label{beh.P1}
f(\xi)\sim D\xi^{-(N-2)/m}, \qquad {\rm as} \ \xi\to0, \ D>0 \ {\rm free \ constant}.
\end{equation}
\end{lemma}
\begin{proof}
The linearization of the system \eqref{PPSyst} near $P_1$ has the matrix
$$M(P_1)=\left(
  \begin{array}{cc}
    \frac{mN-N+2}{m} & 0 \\
    \frac{mN-N+2}{m} & N-2 \\
  \end{array}
\right),$$
with eigenvalues $\lambda_1=(mN-N+2)/m$, $\lambda_2=N-2>0$, thus it is an unstable node. The orbits going out of $P_1$ are such that $Y\sim-(N-2)/m$ and $X\to0$, which after substitution in terms of profiles give
\begin{equation}\label{interm2}
\frac{f'(\xi)}{f(\xi)}\sim-\frac{N-2}{m\xi},
\end{equation}
leading by integration to \eqref{beh.P1}. Assume for contradiction that \eqref{interm2} holds true as $\xi\to\xi_0\in(0,\infty)$ or as $\xi\to\infty$. In both cases we obtain that $f(\xi)\to\infty$ from the fact that $X=0$ at $P_1$, which contradicts \eqref{beh.P1}. It thus remains that both \eqref{interm2} and \eqref{beh.P1} hold true as $\xi\to0$, as claimed.
\end{proof}
With this local analysis in mind, the good profiles are all contained in the unique orbit going out of $P_0$, which becomes our main object of study.

\subsection{Analysis at the infinity of the phase plane}\label{subsec.infinity}

The local analysis of the phase plane of the system \eqref{PPSyst} has to be completed with the analysis of the critical points at infinity. To this end, we pass to the Poincar\'e sphere following \cite[Section 3.10]{Pe}. We introduce new variables
$$
X=\frac{\overline{X}}{W}, \qquad Y=\frac{\overline{Y}}{W}
$$
and it follows that the critical points at infinity lie on the equator of the Poincar\'e sphere, that is, at points $(\overline{X},\overline{Y},W)$ such that $\overline{X}^2+\overline{Y}^2=1$ and $W=0$. If we denote by $P(X,Y)$ and $Q(X,Y)$ the right hand sides of the two equations of the system \eqref{PPSyst}, we notice that the highest order of both $P(X,Y)$ and $Q(X,Y)$ is quadratic, since $(m-p)/(m-1)<2$. Thus, following the theory in \cite[Section 3.10]{Pe}, we let
$$
P^*(\overline{X},\overline{Y},W)=W^2P\left(\frac{\overline{X}}{W},\frac{\overline{Y}}{W}\right), \qquad Q^*(\overline{X},\overline{Y},W)=W^2Q\left(\frac{\overline{X}}{W},\frac{\overline{Y}}{W}\right)
$$
and the critical points at infinity are obtained as solutions to the following equation
$$
\overline{X}Q^*(\overline{X},\overline{Y},W)-\overline{Y}P^*(\overline{X},\overline{Y},W)=0.
$$
We find after straightforward calculations that
\begin{equation}\label{interm0}
\begin{split}
\overline{X}Q^*(\overline{X},\overline{Y},W)&-\overline{Y}P^*(\overline{X},\overline{Y},W)=-m\overline{X}\overline{Y}^2-(N-2)\overline{X}\overline{Y}W+2\overline{X^2}W\\
&-(m-1)\overline{X}^2\overline{Y}-K\overline{X}^{(2m-p-1)/(m-1)}W^{(m+p-2)/(m-1)}\\&-2\overline{X}\overline{Y}W+(m-1)\overline{X}\overline{Y}^2=0.
\end{split}
\end{equation}
Evaluating the term in \eqref{interm0} at $W=0$ and recalling that $m+p-2>0$ we obtain that
$$
\overline{X}\overline{Y}(\overline{Y}+(m-1)\overline{X})=0,
$$
leading thus (together with the fact that $\overline{X}^2+\overline{Y}^2=1$) to the following four critical points at infinity
$$
Q_1=(1,0,0), \ Q_2=(0,1,0), \ Q_3=(0,-1,0), \ Q_4=\left(\frac{1}{\sqrt{1+(m-1)^2}},-\frac{m-1}{\sqrt{1+(m-1)^2}},0\right).
$$
We perform the local analysis of these points one by one below.
\begin{lemma}[Local analysis near $Q_1$]\label{lem.Q1}
The critical point $Q_1$ behaves like a stable node at infinity. The orbits entering it contain all the profiles with interface of Type II, that is, with a local behavior given by \eqref{TypeII}.
\end{lemma}
\begin{proof}
The flow in a neighborhood of the critical point $Q_1$ is topologically equivalent, according to \cite[Theorem 2, Section 3.10]{Pe} to the flow near the origin for the system
\begin{equation}\label{PPSyst2}
\left\{\begin{array}{ll}\dot{y}=2z-(m-1)y-Nyz-y^2-Kz^{(m+p-2)/(m-1)}, \\ \dot{z}=(m-1)yz-2z^2\end{array}\right.
\end{equation}
Due to the fact that $m+p-2<m-1$, we need to perform a further change of variable and set $w=z^{(m+p-2)/(m-1)}$ to get a new system
\begin{equation}\label{PPSyst2bis}
\left\{\begin{array}{ll}\dot{y}=2w^{(m-1)/(m+p-2)}-(m-1)y-Nyw^{(m-1)/(m+p-2)}-y^2-Kw, \\ \dot{w}=(m+p-2)yw-\frac{2(m+p-2)}{m-1}w^{1+(m-1)/(m+p-2)},\end{array}\right.
\end{equation}
where we can linearize near the origin. This linearization has the matrix
$$M(Q_1)=\left(
  \begin{array}{cc}
    -(m-1) & -K \\
    0 & 0 \\
  \end{array}
\right),$$
with eigenvalues $\lambda_1=-(m-1)<0$ and $\lambda_2=0$. We thus have a stable manifold and center manifolds (that may not be unique). The analysis of the center manifold (following \cite[Section 2.12]{Pe}) readily gives that the equation of any center manifold of the system \eqref{PPSyst2bis} in a neighborhood of the point $(y,w)=(0,0)$ is (in a first order approximation)
$$
-(m-1)y-Kw=o(|(y,w)|), \qquad {\rm that \ is } \ \qquad y(w)=-\frac{K}{m-1}w+o(w)
$$
and the flow on the center manifold is given by
$$
\dot{w}=-\frac{K(m+p-2)}{m-1}w^2+o(w),
$$
hence the orbits that are tangent to any center manifold enter the critical point $Q_1$. Thus, the critical point $Q_1$ behaves as a stable node. The orbits entering $Q_1$ tangent to its center manifolds contain profiles such that $y/w\sim-K/(m-1)$, hence
\begin{equation}\label{interm3}
y\sim-\frac{K}{m-1}z^{(m+p-2)/(m-1)}.
\end{equation}
We go back to our initial variables $(X,Y)$ taking into account that $y=Y/X$ and $z=1/X$, thus \eqref{interm3} leads to
\begin{equation}\label{interm.new}
Y\sim-\frac{K}{m-1}X^{(1-p)/(m-1)},
\end{equation}
which after a substitution by the formulas in \eqref{change} gives
\begin{equation*}
(f^{1-p})'(\xi)\sim-\frac{K(1-p)}{m-1}\xi^{(3-2p-m)/(m-1)}
\end{equation*}
and after integration
\begin{equation}\label{interm4}
f(\xi)^{1-p}\sim D-\frac{K}{2}\xi^{2(1-p)/(m-1)}.
\end{equation}
Assume now for contradiction that \eqref{interm4} holds true as $\xi\to0$. In this case, since $X\to\infty$, we readily infer from the definition of $X$ in \eqref{change} that $f(\xi)\to0$, leading to $D=0$ and a contradiction with the required positivity of the profiles $f(\xi)$ near $\xi=0$. It is obvious that \eqref{interm4} cannot hold true as $\xi\to\infty$, hence we are left with \eqref{interm4} as $\xi\to\xi_0\in(0,\infty)$. This easily leads to the behavior of interface of Type II given in \eqref{TypeII}.
\end{proof}
\begin{lemma}[Local analysis near $Q_2$ and $Q_3$]\label{lem.Q23}
The critical points $Q_2$ and $Q_3$ are, respectively, an unstable node and a stable node. The orbits entering them contain profiles with a change of sign at some finite point $\xi_0\in(0,\infty)$ with the following behavior:

$\bullet$ the profiles contained in the orbits going out of $Q_2$ change sign from negative to positive at $\xi_0\in[0,\infty)$ with behavior
$$
f(\xi)\sim C(\xi-\xi_0)^{1/m}, \qquad {\rm as} \ \xi\to\xi_0, \ \xi>\xi_0.
$$

$\bullet$ the profiles contained in the orbits entering $Q_3$ change sign from positive to negative at $\xi_0\in(0,\infty)$ with behavior
$$
f(\xi)\sim C(\xi_0-\xi)^{1/m}, \qquad {\rm as} \ \xi\to\xi_0, \ \xi<\xi_0.
$$
\end{lemma}
\begin{proof}
We deduce from \cite[Theorem 2, Section 3.10]{Pe} that the flow of the system \eqref{PPSyst} in a neighborhood of the critical points $Q_2$ and $Q_3$ is topologically equivalent with the flow in a neighborhood of the origin of the following system
\begin{equation}\label{PPSyst3}
\left\{\begin{array}{ll}\pm\dot{x}=-x-Nxz+2x^2z-(m-1)x^2-Kx^{1+(m-p)/(m-1)}z^{(m+p-2)/(m-1)}\\ \pm\dot{z}=-mz-(N-2)z^2+2xz^2-(m-1)xz-Kx^{(m-p)/(m-1)}z^{1+(m+p-2)/(m-1)},\end{array}\right.
\end{equation}
where the new variables are expressed in terms of the initial variables by $x=X/Y$, $z=1/Y$. The choice of the signs plus or minus in the system \eqref{PPSyst3} is determined by the direction of the flow in a neighborhood of the critical points. Since $Q_2=(0,1,0)$ is the point where $Y\to\infty$ and $Y/X\to\infty$ and $Q_3=(0,-1,0)$ is the point with $Y\to-\infty$ and $Y/X\to-\infty$, we readily get from the second equation of the system \eqref{PPSyst} that $\dot{Y}<0$ in a neighborhood of both $Q_2$ and $Q_3$, thus the minus sign corresponds to $Q_2$ and the plus sign corresponds to $Q_3$ since the direction of the flow near both points is always from right to left. Noticing at a formal level that the two terms with fractional powers in \eqref{PPSyst3} are of lower orders near the origin than the linear ones and neglecting them, we find that the linearization of the system \eqref{PPSyst3} in a neighborhood of the origin has eigenvalues 1 and $m$ if the minus sign is taken, thus $Q_2$ is an unstable node and its opposite $Q_3$ is a stable node. A rigorous approach for this local analysis is done straightforwardly by introducing the change of variable $w=z^{(m+p-2)/(m-1)}$ and analyze the newly obtained system, we do not enter into details as it is rather similar to the analysis done for the point $Q_1$ in coordinates $(y,w)$. Finally, in order to derive the behavior of the profiles contained in the orbits near these points, we infer from \eqref{PPSyst3} that $dx/dz\sim x/mz$ or equivalently after one step of integration
\begin{equation}\label{interm5}
X^{m/(m-1)}\sim CY, \qquad C \ {\rm free \ constant}
\end{equation}
which after a substitution by \eqref{change} leads to
$$
f^{m-1}(\xi)f'(\xi)\sim C\xi^{(m+1)/(m-1)}
$$
or equivalently after integration
\begin{equation}\label{interm6}
f^m(\xi)\sim C_2+C_1\xi^{2m/(m-1)}, \qquad C_1=\frac{m-1}{2C}, \ C_2 \ {\rm free \ constant},
\end{equation}
where $C$ is the constant in \eqref{interm5}. Assume now for contradiction that \eqref{interm5} and \eqref{interm6} hold true as $\xi\to0$ or as $\xi\to\infty$. Since we analyze the neighborhood of points where $Y\to\infty$ or $Y\to-\infty$, we infer from \eqref{interm5} that also $X\to\infty$ and that the sign of the free constant $C$ in \eqref{interm5} is fixed since $X>0$: $C>0$ if we are in a neighborhood of $Q_2$ and $C<0$ if we are in a neighborhood of $Q_3$. In the former case, if \eqref{interm6} holds true as $\xi\to0$, we get on the one hand from \eqref{interm6} that $f(\xi)\sim C_2$ as $\xi\to0$, but on the other hand $X\to+\infty$ readily leads to $f(\xi)\to0$ as $\xi\to0$, that is $C_2=0$ and we obtain a single connection with a change of sign at $\xi=0$ and behavior $f(\xi)\sim C_1\xi^{1/m}$ as $\xi\to0$, $C_1>0$. In the latter case, if \eqref{interm6} holds true as $\xi\to\infty$, we readily get from the positivity of $f(\xi)$ that necessarily $C_1>0$, thus $C>0$ in \eqref{interm5}. We next deduce from \eqref{interm6} that
$$
X(\xi)=\xi^2f(\xi)^{1-m}\sim\xi^2\xi^{-2m}=\xi^{2(1-m)}\to0, \ {\rm as} \ \xi\to\infty,
$$
and a contradiction to \eqref{interm5}. Thus, there are no connections satisfying \eqref{interm5} and \eqref{interm6} as $\xi\to\infty$. It follows that all the rest of the orbits satisfy \eqref{interm5} and \eqref{interm6} as $\xi\to\xi_0\in(0,\infty)$. For orbits entering $Q_3$, we have $C<0$ in \eqref{interm5}, thus $C_1<0$ in \eqref{interm6} and we are forced to restrict the analysis to $C_2>0$, thus readily obtaining the claimed behavior. For the orbits going out of $Q_2$ we can take a priori any constant $C_2$, since $C_1>0$ in this case, but since $X(\xi)\to\infty$ as $\xi\to\xi_0\in(0,\infty)$ we easily find that $C_2<0$ (in order for $f(\xi)$ to have a change of sign at $\xi=\xi_0$), as claimed.
\end{proof}
\begin{lemma}[Local analysis near $Q_4$]\label{lem.Q4}
The critical point at infinity $Q_4$ is a saddle point. The orbits entering it from the finite part of the phase plane contain all the profiles with interfaces of Type I, that is, with a local behavior given by \eqref{TypeI}.
\end{lemma}
\begin{proof}
With the same change of variables as in Lemma \ref{lem.Q1}, the flow of the system \eqref{PPSyst} in a neighborhood of $Q_4$ is topologically equivalent to the flow of the system \eqref{PPSyst2} in a neighborhood of the critical point $(y,z)=(-(m-1),0)$. The linearization of the system \eqref{PPSyst2} near the critical point $(y,z)=(-(m-1),0)$ has the matrix
$$M(Q_4)=\left(
  \begin{array}{cc}
    (m-1) & -K \\
    0 & -(m-1)(m+p-2)\\
  \end{array}
\right),$$
with eigenvalues $\lambda_1=(m-1)>0$ and $\lambda_2=-(m-1)(m+p-2)<0$, thus $Q_4$ is a saddle point. The orbit entering this point contains profiles whose behavior is given by $Y/X\sim-(m-1)$, whence by integration
\begin{equation}\label{interm7}
f^{m-1}(\xi)\sim C_3-\frac{(m-1)^2}{2}\xi^2, \qquad C_3>0 \ {\rm free \ constant}.
\end{equation}
Assume for contradiction that \eqref{interm7} holds true as either $\xi\to0$ or $\xi\to\infty$. In the former, since $X(\xi)\to\infty$ as $\xi\to0$ (as we are analyzing the point $Q_4$ whose orbits have $X\to\infty$ and $Y\to-\infty$), we obtain that $f(\xi)\to0$ as $\xi\to0$, thus $C_3=0$ and a contradiction with the positivity of $f(\xi)$. The latter is obviously impossible, as it contradicts the positivity of $f(\xi)$. Thus, \eqref{interm7} holds true as $\xi\to\xi_0$ for some finite $\xi_0\in(0,\infty)$ and we get the behavior of interface of Type I given by \eqref{TypeI}.
\end{proof}
We are now ready to pass to the global analysis of the connections in the phase plane associated to the system \eqref{PPSyst}.

\subsection{Global analysis with $K>0$ small}\label{subsec.small}

As it follows from the local analysis done in Subsections \ref{subsec.finite} and \ref{subsec.infinity}, we are interested in proving the existence and uniqueness of a connection between the saddle points $P_0$ and $Q_4$. The proof of its existence is based on a \emph{shooting method} in the system \eqref{PPSyst} with respect to the parameter $K$ defined in \eqref{param}. The first part of the shooting method, corresponding to $K>0$ sufficiently small, is performed in this section. We analyze the isoclines of the system \eqref{PPSyst}, that is, the curves where $\dot{X}=0$, respectively $\dot{Y}=0$. The former is just the horizontal line $Y=2/(m-1)$ together with the $Y$ axis. The latter is defined by
$$
-mY^2-[(N-2)+(m-1)X]Y+2X-KX^{(m-p)/(m-1)}=0,
$$
which defines two branches of $Y(X)$ as follows:
\begin{equation}\label{branches}
Y_1(X)=\frac{-(N-2)-(m-1)X+\sqrt{\Delta(X)}}{2m}, \ Y_2(X)=\frac{-(N-2)-(m-1)X-\sqrt{\Delta(X)}}{2m},
\end{equation}
where
\begin{equation}\label{delta}
\Delta(X)=(m-1)^2X^2+2(mN+2m-N+2)X+(N-2)^2-4KmX^{(m-p)/(m-1)}.
\end{equation}
The two branches intersect when $\Delta(X)=0$. We gather in the following rather long Lemma the properties of the two branches introduced in \eqref{branches} for $K>0$ sufficiently small.
\begin{lemma}\label{lem.small}
There exists $K_0>0$ sufficiently small such that for any $K\in(0,K_0)$, the following properties hold true:

(a) There is no intersection between the two branches $Y_1(X)$ and $Y_2(X)$ and $Y_1(X)<2/(m-1)$ for any $X\geq0$.

(b) There exists a unique $X_0(K)>0$ (depending on $K$) such that $Y_1(X)>0$ for $X\in(0,X_0)$ and $Y_1(X)<0$ for $X>X_0$.

(c) The branch $Y_1(X)$ is strictly decreasing with respect to $X$ for $X\in(X_0,\infty)$.
\end{lemma}
\begin{proof}
(a) The two branches in \eqref{branches} may intersect at points $X>0$ where $\Delta(X)=0$. We observe that in the limit case $K=0$ we have $\Delta(X)>(N-2)^2>0$, and this uniform bound from below suggests us that also for $K>0$ sufficiently close to zero $\Delta(X)$ should remain strictly positive. Below we prove it rigorously. Taking into account the expression of $\Delta(X)$ in \eqref{delta}, we find that $\Delta(X)>0$ for $X>0$ sufficiently small (since $(N-2)^2$ is dominating near $X=0$). More precisely, on the one hand we have
\begin{equation}\label{interm8}
\Delta(X)>(N-2)^2-4KmX^{(m-p)/(m-1)}\geq0, \ \ {\rm if} \ 0\leq X\leq X_1(K):=\left[\frac{(N-2)^2}{4mK}\right]^{(m-1)/(m-p)}
\end{equation}
On the other hand, taking into account that $2>(m-p)/(m-1)$, the term in $X^2$ dominates for $X$ very large:
\begin{equation}\label{interm9}
\Delta(X)>(m-1)^2X^2-4KmX^{(m-p)/(m-1)}\geq0, \ \ {\rm if} \ X>X_2(K):=\left[\frac{4Km}{(m-1)^2}\right]^{(m-1)/(m+p-2)}.
\end{equation}
Taking into account that $m>1$, $p<1$ and $m+p-2>0$ we readily get that $X_1(K)\to\infty$ and $X_2(K)\to0$ as $K\to0$. Thus, there exists $K_0>0$ such that $X_1(K)>X_2(K)$ for any $K\in(0,K_0)$ and we can gather \eqref{interm8} and \eqref{interm9} to obtain that $\Delta(X)>0$ for any $X\geq0$. The fact that $Y_1(X)<2/(m-1)$ for any $X>0$ is easy: if it were some intersection between the two isoclines, any intersection point would be a finite critical point for the system \eqref{PPSyst}, and we know that there is no other finite critical point than $P_0$ and $P_1$.

\medskip

(b) The intersections between the branch $Y_1(X)$ and the axis $Y=0$ take place at points where
$$
\Delta(X)=(N-2)+(m-1)X,
$$
or equivalently after taking squares and simplifying
$$
2(mN+2m-N+2)X-4KmX^{(m-p)/(m-1)}=2(m-1)(N-2)X,
$$
which after obvious simplifications leads to either $X=0$ or the equation $2=KX^{(1-p)/(m-1)}$ that gives
\begin{equation}\label{interm10}
X=X_0(K):=\left(\frac{2}{K}\right)^{(m-1)/(1-p)},
\end{equation}
which is the unique positive intersection, as claimed (and it has an explicit dependence on $K$). Moreover, it is obvious that in a right-neighborhood of $X=0$ we have $Y_1(X)>0$ since for $X<X_0(K)$ we have $2>KX^{(1-p)/(m-1)}$. It then remains that $Y_1(X)<0$ for $X>X_0(K)$.

\medskip

(c) By direct calculation we find
$$
Y_1'(X)=\frac{1}{4m\sqrt{\Delta(X)}}\left[\Delta'(X)-2(m-1)\sqrt{\Delta(X)}\right]
$$
and we have to study the sign of the term in brackets. To this end, we compute
\begin{equation*}
\Delta'(X)^2-4(m-1)^2\Delta(X)=32m(mN-N+2)-16mg(X),
\end{equation*}
where
\begin{equation*}
\begin{split}
g(X)&=K(m-1)(1-p)X^{(m-p)/(m-1)}+\frac{K(m-p)(mN-N+2m+2)}{m-1}X^{(1-p)/(m-1)}\\
&-\frac{K^2m(m-p)^2}{(m-1)^2}X^{2(1-p)/(m-1)}.
\end{split}
\end{equation*}
Recalling the value of $X_0(K)$ introduced in \eqref{interm10} we readily get that
\begin{equation*}
\begin{split}
g(X_0(K))&=2^{(m-p)/(m-1)}(m-1)(1-p)\left(\frac{1}{K}\right)^{(m-1)/(1-p)}\\
&+\frac{2(m-p)(mN-N+2m+2)}{m-1}-\frac{4m(m-p)^2}{(m-1)^2}
\end{split}
\end{equation*}
which tends to infinity as $K\to0$, hence for $K\in(0,K_0)$ sufficiently small we have $g(X_0(K))>2m(mN-N+2)$ and thus $Y_1'(X_0(K))<0$. We then notice that we can write
$$
g(X)=KX^{(1-p)/(m-1)}h(X),
$$
where
\begin{equation*}
h(X)=(m-1)(1-p)X+\frac{(m-p)(mN-N+2m+2)}{m-1}-\frac{Km(m-p)^2}{(m-1)^2}X^{(1-p)/(m-1)}.
\end{equation*}
We have
$$
h'(X)=(m-1)(1-p)-\frac{Km(m-p)^2(1-p)}{(m-1)^3}X^{(2-m-p)/(m-1)},
$$
thus $h(X)$ has a unique maximum point in the half-plane $X>0$, namely at a point $X_3(K)$ given by
$$
X_3(K)^{(2-m-p)/(m-1)}=(1-p)\left[(m-1)^4-Km(m-p)^2\right]
$$
Since $X_3(K)$ tends to a constant as $K\to0$, we obtain that for $K$ sufficiently small $X_3(K)<X_0(K)$, whence $h(X)$ is increasing in the interval $(X_0(K),\infty)$ (recalling that the power $(2-m-p)/(m-1)$ is negative). It then follows that $g(X)$ is also increasing for $X>X_0(K)$, thus $Y_1'(X)<0$ in the same interval, as claimed.
\end{proof}
We thus conclude from Lemma \ref{lem.small} that the two isoclines for $K\in(0,K_0)$ with $K_0>0$ sufficiently small divide the half-plane $\{X\geq0\}$ into four regions as follows:

$\bullet$ the region $\{Y>2/(m-1)\}$, where $dY/dX>0$ along the trajectories, called region (I)

$\bullet$ the region lying between the horizontal line $\{Y=2/(m-1)\}$ and the branch $Y_1(X)$, where $dY/dX<0$ along the trajectories, called region (II).

$\bullet$ the region between the two branches $Y_1(X)$ and $Y_2(X)$, where $dY/dX>0$ along the trajectories, called region (III)

$\bullet$ the region lying below the branch $Y_2(X)$, where $dY/dX<0$ along the trajectories, called region (IV).

The regions are illustrated in Figure \ref{fig1} below.

\begin{figure}[ht!]
  \begin{center}
  \includegraphics[width=11cm,height=7.5cm]{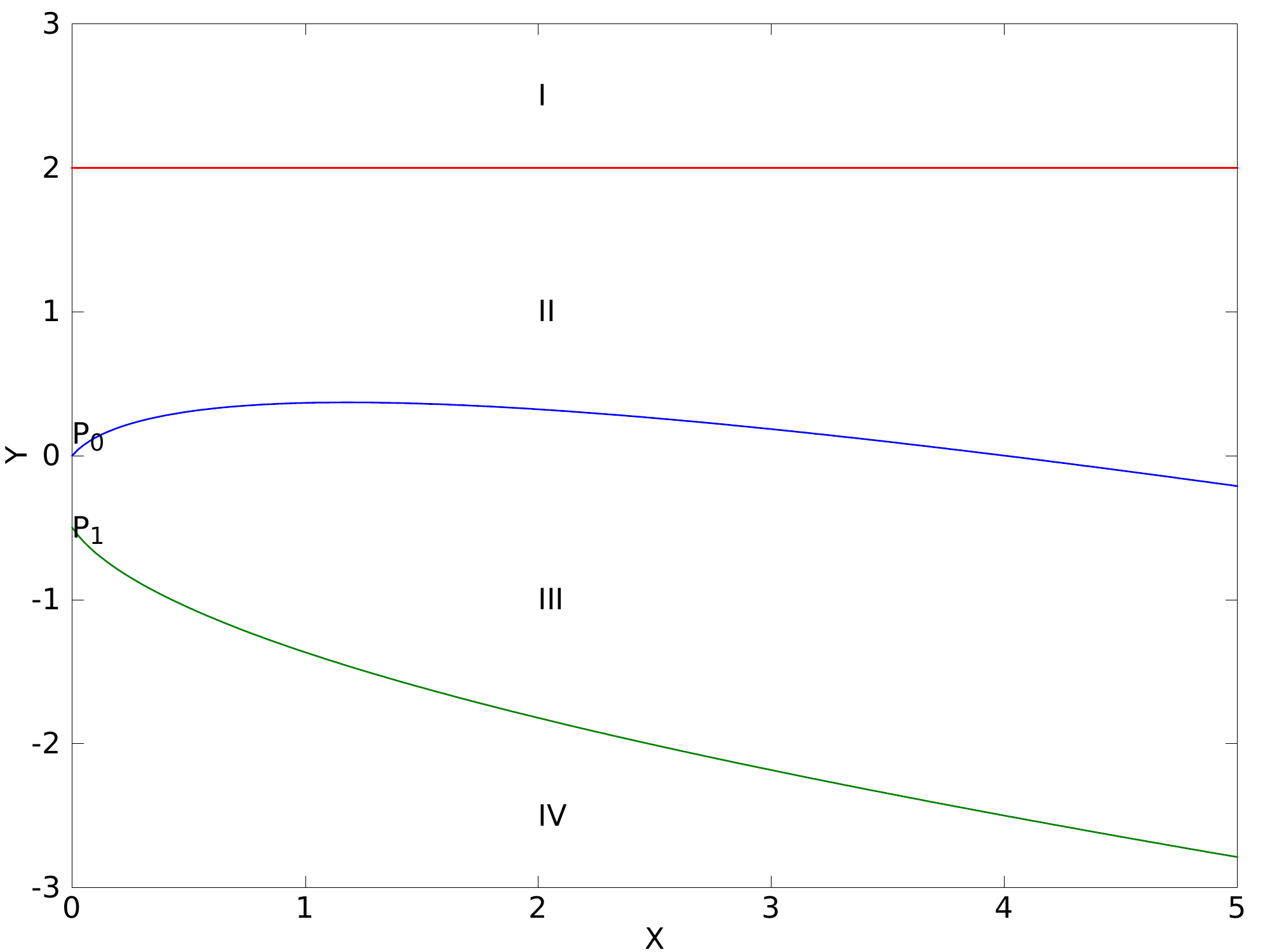}
  \end{center}
  \caption{The four regions in the phase plane separated by the isoclines}\label{fig1}
\end{figure}

We next have to study in the limit of which of these regions lie the critical points at infinity and this knowledge will allow us to establish the endpoint of the unique orbit going out of the saddle point $P_0$.
\begin{proposition}\label{prop.small}
For any $K\in(0,K_0)$ with $K_0>0$ as in the statement of Lemma \ref{lem.small}, the critical point $Q_1$ lie in the limit of region (II), while the critical points $Q_3$ and $Q_4$ lie in the limit of the region (IV). The orbit going out of $P_0$ connects to the point $Q_1$ for any $K\in(0,K_0)$.
\end{proposition}
\begin{proof}
For the critical point $Q_1$ we recall from Lemma \ref{lem.Q1} that the local behavior of the profiles contained in the orbits entering it is given by \eqref{TypeII}, and in terms of the phase plane variables the relation between the components $X$ and $Y$ on the trajectories entering $Q_1$ is given by \eqref{interm.new}. It is on the other hand easy to see that the first order approximation of the branch $Y_2(X)$ is
\begin{equation}\label{interm12}
Y_2(X)\sim-\frac{m-1}{m}X, \qquad {\rm as} \ X\to\infty,
\end{equation}
and we infer from \eqref{interm.new}, \eqref{interm12} and the fact that $(1-p)/(m-1)<1$ in our range of exponents that the critical point $Q_1$ lies strictly above the branch $Y=Y_2(X)$, that is, in the limit as $X\to\infty$ of one of the regions (II) and (III). Since $Q_1$ is an attractor and it is obvious that the orbits cannot enter $Q_1$ directly in region (III) due to the monotonicity of the orbits inside this region (which is increasing in region (III) due to the fact that $dY/dX>0$), we are left with $Q_1$ in the limit of region (II).

For the critical point $Q_3$ it is immediate: since in a neighborhood of $Q_3$ we have $Y/X\to-\infty$, we infer from \eqref{interm12} that $Q_3$ lies below the branch $Y=Y_2(X)$, thus in region (IV). Finally, for the critical point $Q_4$ we know from Lemma \ref{lem.Q4} that $Y/X\to-(m-1)$ as $X\to\infty$, thus $Y\sim-(m-1)X$ as $X\to\infty$. Comparing with the asymptotic behavior of the branch $Y=Y_2(X)$ given in \eqref{interm12} and taking into account that $(m-1)/m<m-1$, we find that $Q_4$ also lies in the limit of region (IV).

We are now in a position to "drive" the orbit going out of $P_0$ and show that it connects to $Q_1$ only by geometric arguments. Indeed, as it goes out of $P_0$ tangent to the eigenvector $(1,2/N)$, it starts in the region $\{X>0,Y>0\}$ and has to do it inside the region (III), that is, below the branch $Y=Y_1(X)$, as this is the region where the trajectories are increasing as functions $Y(X)$ in the plane $(X,Y)$. Then the trajectory increases (as a function $Y(X)$) until it crosses the branch $Y_1(X)$ and then becomes decreasing. Consider now the branch $Y=Y_1(X)$ as a barrier. The normal vector to it has the direction given by the vector $(-Y_1'(X),1)$, thus the direction of the flow of the system \eqref{PPSyst} over this branch is given by the sign of the expression
$$
-Y_1'(X)X(2-(m-1)Y_1(X))
$$
which is positive in the region where $Y_1(X)$ is decreasing. It thus follows that the branch cannot be crossed from region (II) into region (III) once it becomes decreasing (in particular, in the region where $Y<0$, as given by Lemma \ref{lem.small}), hence the orbit stays inside region (II) for any $X$ large and any $K\in(0,K_0)$ for which the statement of Lemma \ref{lem.small} holds true. By the previous arguments and its monotonicity, the orbit has to enter the attractor $Q_1$.
\end{proof}
For the reader's convenience, a plot of the orbits in the phase plane corresponding to the global analysis in this Subsection is given in the first half of Figure \ref{fig3} below.

\subsection{Global analysis with $K>0$ large}\label{subsec.large}

For $K$ sufficiently large the geometric picture changes. It can be proved rather straightforwardly that there are exactly two intersections of the two branches $Y_1(X)$ and $Y_2(X)$ given in \eqref{branches} and that there is a region in between where the branches become complex (since $\Delta(X)<0$ between the two roots giving the intersection points). Thus, we can no longer use a split of the plane into regions of monotonicity as we did in Subsection \ref{subsec.small}. In this case, the approach is based on another barrier in form of a suitable line.
\begin{proposition}\label{prop.large}
There exists $K_1>0$ such that for any $K\in(K_1,\infty)$, the orbit going out of $P_0$ connects to the stable node $Q_3$ at infinity.
\end{proposition}
\begin{proof}
Consider as a barrier for the orbits the straight line
\begin{equation}\label{interm13}
r: (m-1)X+Y=D, \qquad D=\lambda K
\end{equation}
with $\lambda>0$ to be determined later. The normal vector to $r$ is given by $\overline{n}=(m-1,1)$ and the flow of the system \eqref{PPSyst} over the line $r$ is given by the sign of the expression (obtained as the scalar product of $\overline{n}$ with the vector field of the system)
$$
F(X)=(Dm^2-Dm+mN-N+2)X-mD^2-(N-2)D-KX^{(m-p)/(m-1)}.
$$
We notice that $F(0)<0$ and $\lim\limits_{X\to\infty}F(X)=-\infty$, since $(m-p)/(m-1)>1$. Moreover, $F(X)$ has a unique point of absolute maximum at
$$
M(K)=\left[\frac{(m-1)(Dm^2-Dm+mN-N+2)}{(m-p)K}\right]^{(m-1)/(1-p)}.
$$
We obtain by rather tedious but straightforward calculations that
\begin{equation}\label{interm14}
\begin{split}
F(M(K))&=(1-p)\left(\frac{m-1}{K}\right)^{(m-1)/(1-p)}\left[\frac{Dm^2-Dm+mN-N+2}{m-p}\right]^{(m-p)/(1-p)}\\
&-mD^2-(N-2)D
\end{split}
\end{equation}
and recalling that $D=\lambda K$ we can measure the dependence over $K$ to conclude that the dominating term above for $K$ large is $-mD^2=-m\lambda^2K^2$, as the rather tedious positive term in \eqref{interm14} is only of order $K$. There exists thus $K_1$ sufficiently large such that for any $K>K_1$ we have $F(M(K))<0$, whence $F(X)<0$ for any $X\geq0$ and $K\in(K_1,\infty)$. This proves that the line $r$ cannot be crossed from left to right by the orbits of the system \eqref{PPSyst}, thus the orbit going out of $P_0$ for any $K>K_1$ does not cross the line $r$.

We readily get from Lemma \ref{lem.Q4} that the orbit entering the saddle point $Q_4$ enters tangent to the eigenvector corresponding to the negative eigenvalue of the matrix $M(Q_4)$ in the proof of Lemma \ref{lem.Q4}, namely tangent to the line
$$
y=-(m-1)+\frac{K}{(m-1)(m+p-1)}z
$$
in variables $(y,z)$ of the system \eqref{PPSyst2}, which, recalling that $y=Y/X$ and $z=1/X$, translates into the straight line
$$
(m-1)X+Y=\frac{K}{(m-1)(m+p-1)}
$$
in variables $(X,Y)$. Choosing now $\lambda<1/(m-1)(m+p-1)$ in the previous argument, we infer that the orbit going out of $P_0$ with $K>0$ sufficiently large cannot enter the point $Q_4$, thus by monotonicity of the orbit $Y(X)$ and discarding the other points, it has to connect to the critical point $Q_3$.
\end{proof}
We picture in Figure \ref{fig3} an example of the outcome of the global analysis performed in the previous subsections, through a numerical experiment plotting the relevant connections in the phase plane for both $K>0$ sufficiently small and sufficiently large.

\begin{figure}[ht!]
  \begin{center}
  \subfigure[$K>0$ small]{\includegraphics[width=7.5cm,height=6cm]{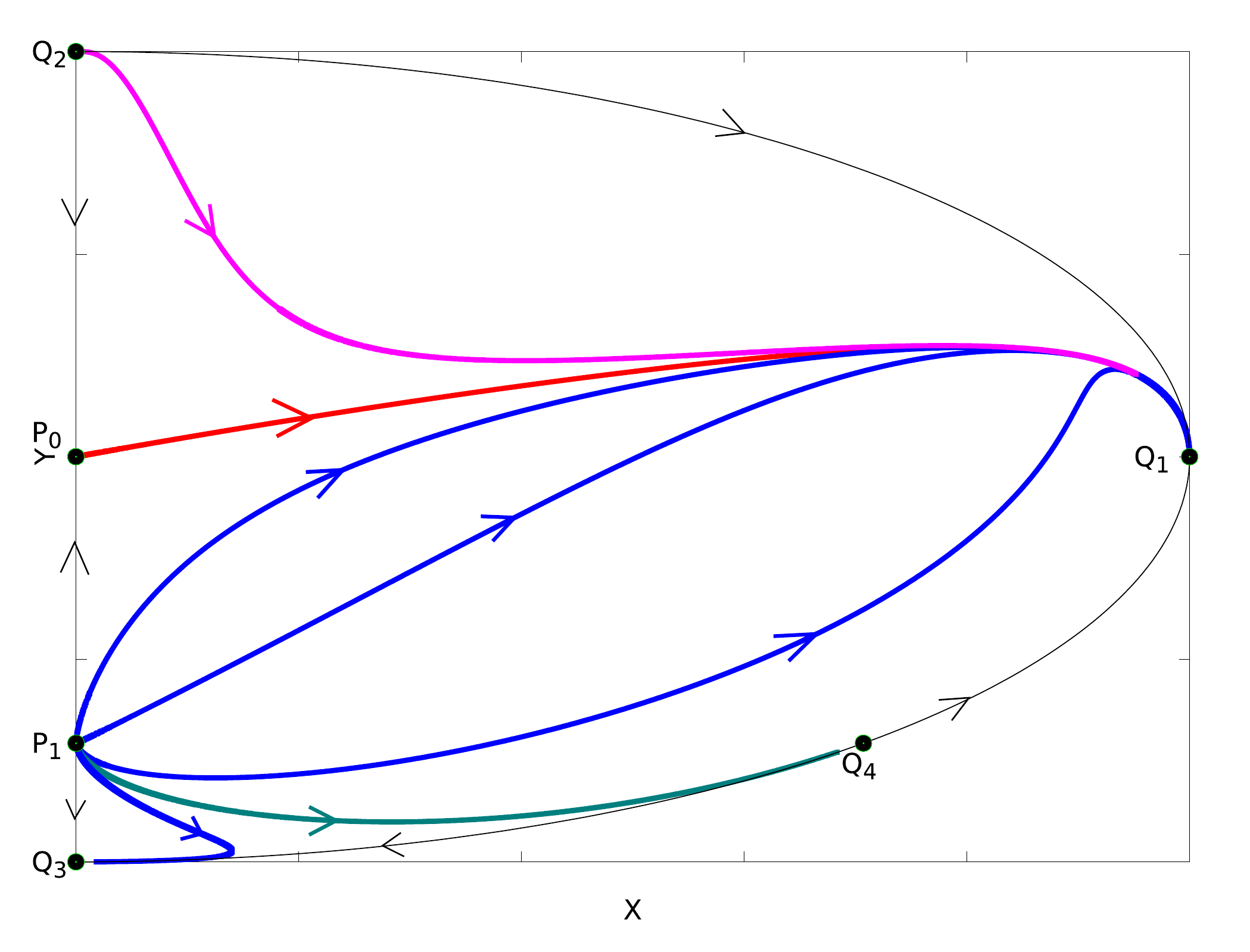}}
  \subfigure[$K>0$ large]{\includegraphics[width=7.5cm,height=6cm]{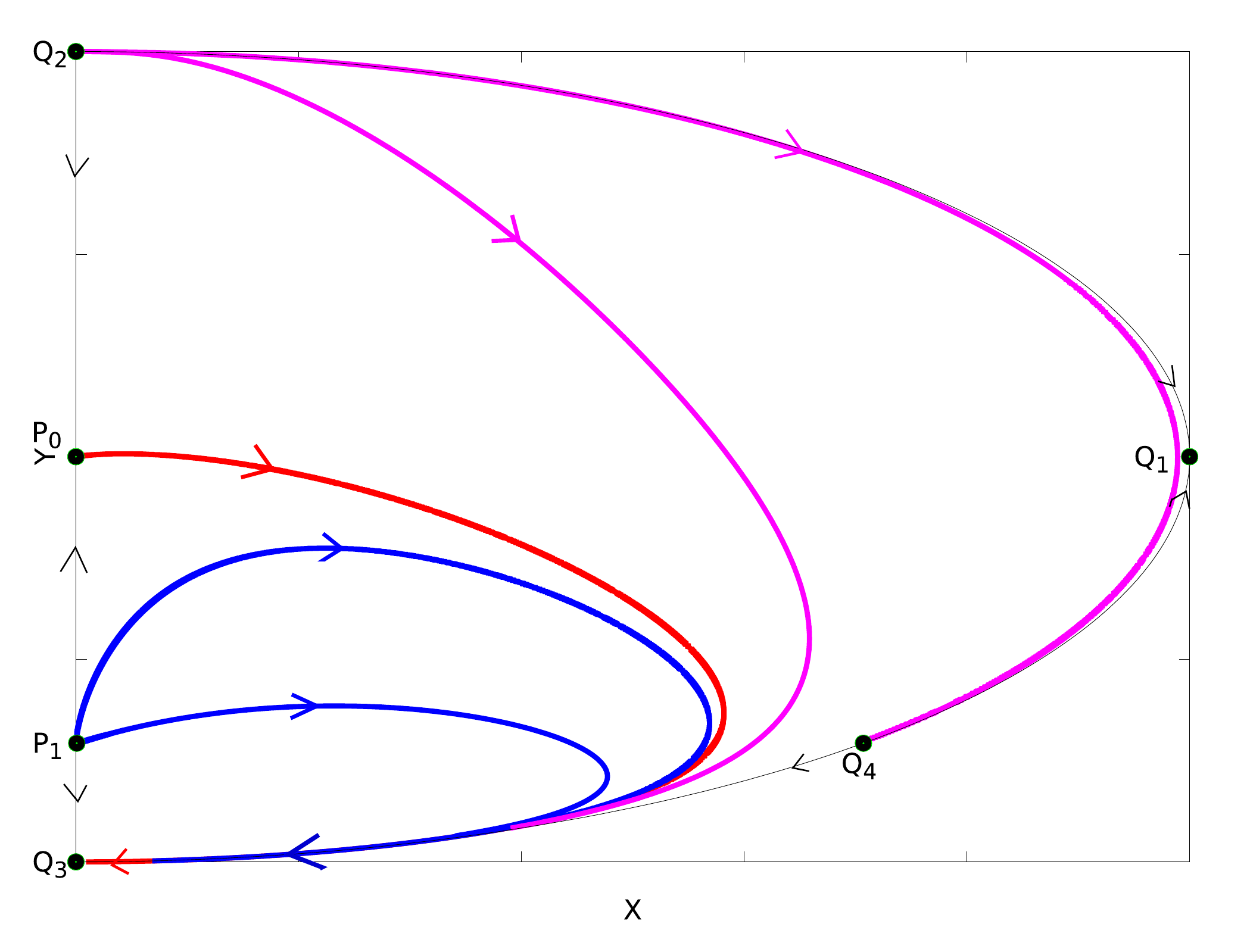}}
  \end{center}
  \caption{Trajectories in the phase plane for different values of $K>0$. Numerical experiment for $m=2$, $p=0.5$ $N=4$, $\sigma=1$ and $K=0.1$, respectively $K=8$}\label{fig3}
\end{figure}

\subsection{Proof of Theorem \ref{th.1}}\label{subsec.proof}

The analysis performed in Subsections \ref{subsec.small} and \ref{subsec.large} allows us to conclude the proof of Theorem \ref{th.1} for $m+p>2$ and $N\geq3$. We split it into the existence and the uniqueness part.

\medskip

\noindent \textbf{Existence.} Let us introduce the following three sets:
\begin{equation*}
\begin{split}
&A=\{K\in(0,\infty):{\rm the \ orbit \ from} \ P_0 \ {\rm connects \ to} \ Q_1\}, \\
&B=\{K\in(0,\infty):{\rm the \ orbit \ from} \ P_0 \ {\rm connects \ to} \ Q_4\}, \\
&C=\{K\in(0,\infty):{\rm the \ orbit \ from} \ P_0 \ {\rm connects \ to} \ Q_3\}.
\end{split}
\end{equation*}
We infer that both $A$ and $C$ are open sets from the fact that $Q_1$ and $Q_3$ are attractors. Moreover, it is obvious that the sets $A$, $B$ and $C$ are disjoint and their union is the interval $(0,\infty)$. Proposition \ref{prop.small} gives that the set $A$ is nonempty and contains an interval of the form $(0,K_0)$, while Proposition \ref{prop.large} gives that the set $C$ is nonempty and contains an interval of the form $(K_1,\infty)$. By a standard topological argument, we deduce that the set $B$ is closed and nonempty, and any element of $B$ corresponds to a connection between $P_0$ and $Q_4$ (which contains good profiles with interface of Type I).

\medskip

\noindent \textbf{Uniqueness.} We use an argument, similar to the one used in \cite{dPS00}, of opposite monotonicity with respect to the parameter $K$ of the orbits going out of $P_0$, respectively entering $Q_4$, to show that the set $B$ is a singleton. On the one hand, as shown at the end of the proof of Proposition \ref{prop.large}, the orbit entering $Q_4$ is tangent to the line
$$
Y(X)=\frac{K}{(m-1)(m+p-1)}-(m-1)X,
$$
which vary in an increasing way with respect to $K$ in a neighborhood of $Q_4$. Since
\begin{equation}\label{interm15}
\frac{dY}{dX}=\frac{-mY^2-(N-2)Y+2X-(m-1)XY-KX^{(m-p)/(m-1)}}{X(2-(m-1)Y)}
\end{equation}
is decreasing with respect to $K$ in the region $\{Y<2/(m-1)\}$, it follows by standard comparison that if $K_1<K_2$ we have $Y_{K_1}(X)<Y_{K_2}(X)$ for any $X>0$ while $Y$ remains below the horizontal line $\{Y=2/(m-1)\}$, where $Y_{K_1}(X)$ and $Y_{K_2}(X)$ are the orbits entering $Q_4$ for $K=K_1$, respectively $K=K_2$. On the other hand, in order to analyze the orbit going out of $P_0$, we have to deduce its second order development near $P_0$. We readily notice from \eqref{interm15} that, since at first order we have $Y\sim 2X/N$, that is, a linear behavior of the orbit $Y(X)$, we can neglect the terms $-mY^2$ and $-(m-1)XY$ which are at least quadratic and infer that
$$
\frac{dY}{dX}=\frac{-(N-2)Y+2X-KX^{(m-p)/(m-1)}}{2X}+o(X^{(1-p)/(m-1)}),
$$
hence we get by integration
$$
Y(X)\sim\frac{2}{N}X-\frac{K(m-1)}{N(m-1)+2(1-p)}X^{(m-p)/(m-1)}+o(X^{(m-p)/(m-1)})
$$
in a small neighborhood of $P_0$. This shows that locally near $P_0$ the orbit vary in a decreasing way with respect to the parameter $K$. A similar argument of comparison based on \eqref{interm15} (and the fact that the orbits from $P_0$ stay always below the line $\{Y=2/(m-1)\}$) gives that along the orbits going out of $P_0$ for parameters $K_1<K_2$ we have $Y_{K_2}(X)<Y_{K_1}(X)$ for any $X>0$. This opposite monotonicity with respect to $K$ along the orbits together with the existence of a connection readily gives the uniqueness of this connection between $P_0$ and $Q_4$.

\medskip

\noindent \textbf{End of the proof.} We thus conclude that there exists $K^*>0$ such that the three sets are $A=(0,K^*)$, $B=\{K^*\}$ and $C=(K^*,\infty)$. Thus, there is a unique good profile with interface of Type I, corresponding to $K=K^*$ and $\alpha=\alpha^*$, where $K^*$ and $\alpha^*$ are linked by \eqref{param}. We also deduce from \eqref{param} and the uniqueness of the orbit going out of $P_0$ for any $K\in(0,K^*)$ that there exists a unique good profile with interface of Type II for any $\alpha>\alpha^*$. Finally, no good profile with interface exists for $\alpha\in(0,\alpha^*)$, corresponding to the range $K>K^*$ which is analyzed in Proposition \ref{prop.large}.

\section{Analysis of the range $m+p=2$ in dimension $N\geq3$}\label{sec.equal}

This section is devoted to the special case when $m+p=2$, whose main effect is that the term $X^{(m-p)/(m-1)}$ in the second equation \eqref{PPSyst} becomes $X^2$, thus the dynamical system becomes quadratic. Since the analysis of the critical points $P_0$ and $P_1$ is totally similar, the differences begin with the analysis of the critical points at infinity.

\subsection{Analysis of the critical points at infinity}\label{subsec.infequal}

We notice that, by letting $W=0$ in the right hand side of \eqref{interm0}, we are left with a new term that was negligible in the range $m+p>2$, thus the equation satisfied by the critical points at infinity on the Poincar\'e sphere becomes
\begin{equation}\label{interm16}
-\overline{X}\left[\overline{Y}^2+(m-1)\overline{X}\overline{Y}+K\overline{X}^2\right]=0.
\end{equation}
If $\overline{X}=0$ we find again the critical points $Q_2=(0,1,0)$ and $Q_3=(0,-1,0)$ on the Poincar\'e sphere and the local analysis is similar to the one in Lemma \ref{lem.Q23}, as it can be noticed by an inspection of the proof. Looking for critical points with $\overline{X}\neq0$, we set $\overline{Y}=\lambda\overline{X}$ and obtain that on the one hand \eqref{interm16} gives
\begin{equation}\label{interm16bis}
\lambda^2+(m-1)\lambda+K=0
\end{equation}
and on the other hand, the condition $\overline{X}^2+\overline{Y}^2=1$ translates into $(\lambda^2+1)\overline{X}^2=1$. Equation \eqref{interm16bis} has real solutions
\begin{equation}\label{y12}
y_{1,2}=\frac{-(m-1)\pm\sqrt{(m-1)^2-4K}}{2},
\end{equation}
provided $0<K\leq(m-1)^2/4$. We thus obtain two critical points at infinity that we relabel as $Q_1$ and $Q_4$
$$
Q_1=\left(\frac{1}{\sqrt{1+y_1^2}},\frac{y_1}{\sqrt{1+y_1^2}},0\right), \qquad Q_4=\left(\frac{1}{\sqrt{1+y_2^2}},\frac{y_2}{\sqrt{1+y_2^2}},0\right),
$$
where $y_2<y_1<0$ are defined in \eqref{y12}, which are different for $0<K<(m-1)^2/4$, coincide for $K=(m-1)^2/4$ and disappear for $K>(m-1)^2/4$. We perform next the local analysis near these points, provided $0<K<(m-1)^2/4$.
\begin{lemma}\label{lem.Q14}
For $K\in(0,(m-1)^2/4)$ the critical point $Q_1$ is a stable node and the critical point $Q_4$ is a saddle point. The profiles contained in the orbits entering them have an interface at some $\xi_0\in(0,\infty)$ with the local behavior
\begin{equation}\label{beh.Q14}
f(\xi)\sim\left(D+\frac{(m-1)y_i}{2}\xi^2\right)_{+}^{1/(m-1)}, \qquad D>0 \ {\rm free \ constant}
\end{equation}
where $y_1$ corresponds to the point $Q_1$ and $y_2$ corresponds to the point $Q_4$.
\end{lemma}
\begin{proof}
In order to study the local behavior of the system \eqref{PPSyst} in a neighborhood of these points, we use again the system \eqref{PPSyst2}, which in our case becomes
\begin{equation}\label{PPSyst.equal}
\left\{\begin{array}{ll}\dot{y}=2z-(m-1)y-Nyz-y^2-K, \\ \dot{z}=(m-1)yz-2z^2\end{array}\right.
\end{equation}
and the two critical points in these variables read $Q_1=(y_1,0)$ and $Q_2=(y_2,0)$. The linearization of the system \eqref{PPSyst.equal} in a neighborhood of these points have the matrices
$$
M(Q_1)=\left(
  \begin{array}{cc}
    -(m-1)-2y_1 & 2-Ny_1 \\
    0 & (m-1)y_1 \\
  \end{array}
\right), \ M(Q_4)=\left(
  \begin{array}{cc}
    -(m-1)-2y_2 & 2-Ny_2 \\
    0 & (m-1)y_2 \\
  \end{array}
\right),
$$
thus, recalling that $y_2<y_1<0$ and noticing that
$$
-(m-1)-2y_i=\mp\sqrt{(m-1)^2-4K},
$$
we readily get that $Q_1$ is a stable node and $Q_4$ is a saddle point. The orbits entering these points are characterized by the fact that $Y/X\sim y_i$, which in terms of profiles leads to $(f^{m-2}f')(\xi)\sim y_i\xi$ and we obtain the local behavior \eqref{beh.Q14} by integration. The fact that the local behavior is taken as $\xi\to\xi_0$ finite is proved similarly as in the proofs of Lemmas \ref{lem.Q1} and \ref{lem.Q4} and we omit the details here.
\end{proof}

\noindent \textbf{Remark.} When $m+p=2$ there is a single type of interface behavior and the only difference between the two types of profiles is actually the constant multiplying $\xi^2$ in \eqref{beh.Q14}.

\subsection{Global analysis for $m+p=2$}\label{subsec.globalequal}

The behavior of the connections in the phase plane of the system \eqref{PPSyst} when $m+p=2$ can be performed by adapting conveniently the similar parts of the analysis done for $m+p>2$ in Section \ref{sec.large}, but the specific relation between the exponents allows for some simplifications. First of all, it is very easy to see that for any $K>(m-1)^2/4$ the orbit going out of $P_0$ connects to $Q_3$, since there are no other critical points where it can connect and no limit cycles are allowed since $Y(X)$ is decreasing for $X$ large. It remains to analyze the orbits in the more interesting range $0<K\leq(m-1)^2/4$.
\begin{proposition}\label{prop.equal}
For any $K\in(0,(m-1)^2/4)$ the orbit going out of $P_0$ enters the stable node $Q_1$.
\end{proposition}
\begin{proof}
We work with the system \eqref{PPSyst.equal}, which is obtained by the change of variables $y=Y/X$, $z=1/X$ and where the critical points $Q_1$ and $Q_4$ lie in the finite part and the points $P_0$ and $P_1$ in the infinite part. Since the orbit going out of $P_0$ is locally tangent to the eigenvector $(1,2/N)$, the critical point at infinity in the system \eqref{PPSyst.equal} that matches to $P_0$ has in the limit the components $y=2/N$ and $z=+\infty$. We study the isoclines of the system \eqref{PPSyst.equal}. The first and most important one is the curve where the first equation in \eqref{PPSyst.equal} vanishes, with
\begin{equation}\label{interm17}
z(y)=\frac{y^2+(m-1)y+K}{2-Ny}, \qquad z'(y)=\frac{-Ny^2+KN+2(m-1)+4y}{(2-Ny)^2},
\end{equation}
which is composed in the region $z>0$ by two branches starting each of them at the critical points $(y_2,0)$, respectively $(y_1,0)$. It is a simple exercise of calculus to notice that, in the range $y<2/N$, the equation $z'(y)=0$ has a single solution (a minimum point lying in the interval $(y_2,y_1)$ that can be made explicit) and thus the branch starting in the region $\{z>0\}$ from the point $(y_2,0)$ has $z(y)$ decreasing, while the branch starting in the region $\{z>0\}$ from the point $(y_1,0)$ has $z(y)$ increasing, until reaching the vertical asymptote $y=2/N$. The second isocline is the straight line $z=(m-1)y/2$ which does not intersect any of the two previous branches as it can be easily seen by the non-existence of a further finite critical point in the phase plane associated to the system \eqref{PPSyst.equal}. We thus conclude that the isoclines divide the half-plane $\{z\geq0\}$ into four regions according to the sign of
$$
\frac{dz(y)}{dy}=\frac{(m-1)yz-2z^2}{2z-(m-1)y-y^2-Nyz-K},
$$
namely:

$\bullet$ a region (I) at the left of the decreasing branch starting from $(y_2,0)$, where $dz(y)/dy>0$

$\bullet$ a region (II) between the two branches in the region $\{z>0\}$ of the function $z(y)$, where $dz(y)/dy<0$

$\bullet$ a region (III) between the increasing branch of $z(y)$ starting from the point $(y_1,0)$ and the straight line $z=(m-1)y/2$, where $dz(y)/dy>0$

$\bullet$ a region (IV) corresponding to the region $\{z<(m-1)y/2\}$, where $dz(y)/dy<0$.

For the easiness of the reading, the regions (I), (II), (III), (IV) above are represented in Figure \ref{fig2}.

\begin{figure}[ht!]
  \begin{center}
  \includegraphics[width=11cm,height=7.5cm]{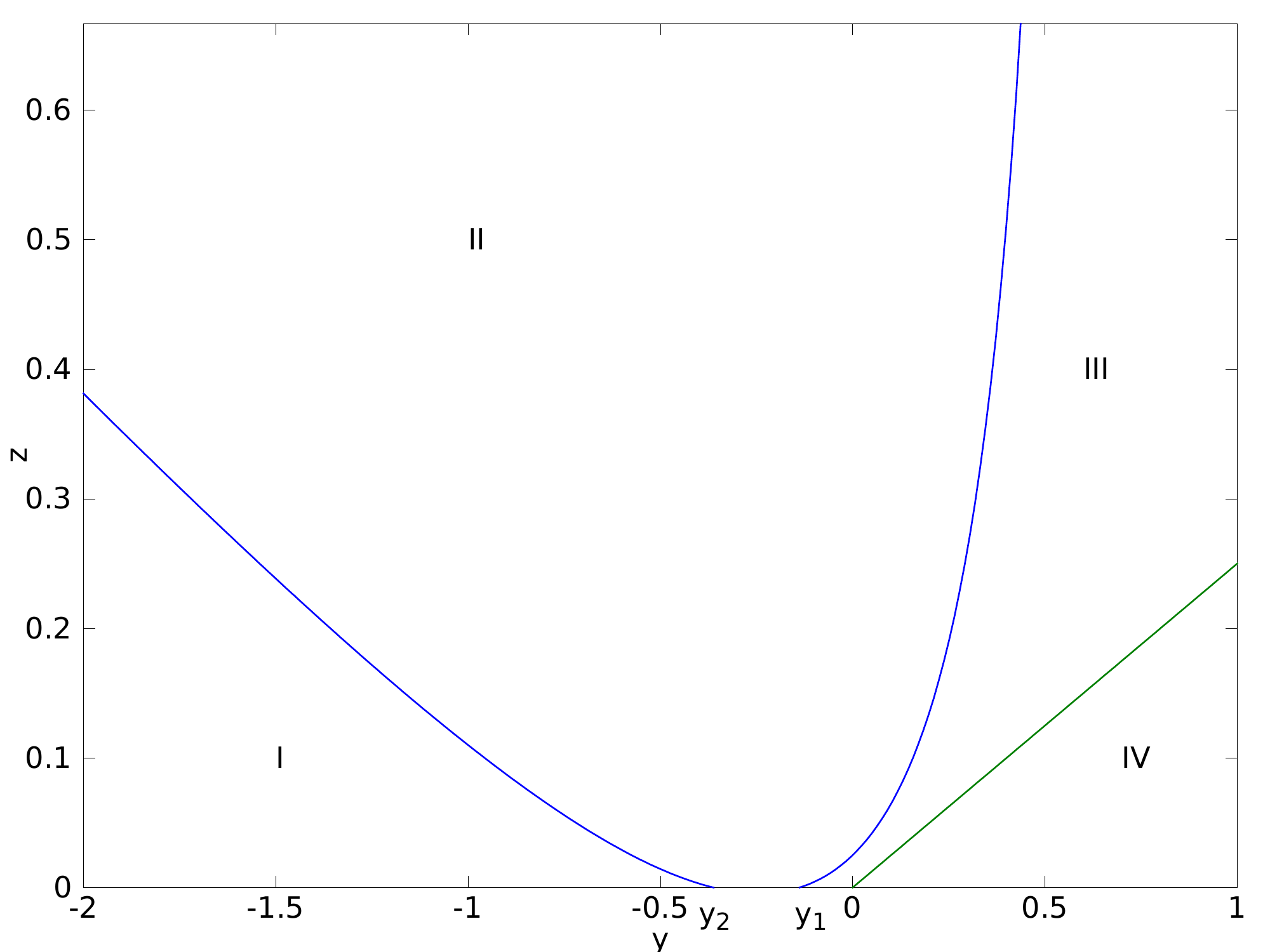}
  \end{center}
  \caption{The regions in the phase plane associated to the system \eqref{PPSyst.equal}}\label{fig2}
\end{figure}

From this monotonicity along the trajectories of the system \eqref{PPSyst.equal} we deduce that the unique orbit entering the saddle point $Q_4$, corresponding in our study to the point $(y_2,0)$, has to enter the point through the region (II), while the unique orbit going out of $P_0$ has to enter necessarily the region (III) above. Thus the two points are separated at least in a neighborhood of the two critical points by the increasing branch of the isocline starting from $(y_2,0)$ in \eqref{interm17}. Considering this isocline of equation $z(y)-z=0$, the normal vector has direction $(z'(y),-1)$ and the flow of the system \eqref{PPSyst.equal} over it is given by the sign of the expression
$$
-\dot{z}+z'(y)\dot{y}=-(m-1)yz+2z^2=2z\left(z-\frac{m-1}{2}y\right)>0,
$$
thus the branch cannot be crossed from right to left since $z'(y)>0$ on it. This proves that the orbit going out of $P_0$ must always enter the point $Q_1$ (identified in our new variables as $(y_1,0)$) through the region (III), while the orbit entering the saddle point $Q_4$ (identified as $(y_2,0)$) must come from the node $P_1$, and this holds true for any $K\in(0,(m-1)^2/4)$, ending the proof.
\end{proof}
We are now in a position to prove Theorem \ref{th.1bis}.
\begin{proof}[Proof of Theorem \ref{th.1bis}]
We have proved in Proposition \ref{prop.equal} that for any $K\in(0,(m-1)^2/4)$ there exists a unique orbit going out of $P_0$ and entering the critical point $Q_1$, corresponding to a unique good profile such that $f(0)=1$ and having an interface behavior given locally by \eqref{beh.Q14} (with $y=y_1$), corresponding qualitatively to an interface of Type I. On the other hand, the set
$$
C=\{K\in(0,\infty):{\rm the \ orbit \ from} \ P_0 \ {\rm connects \ to} \ Q_3\}
$$
is an open set containing the interval $((m-1)^2/4,\infty)$ but which does not contain any point $K<(m-1)^2/4$. It is thus obvious that
$$
C=\left(\frac{(m-1)^2}{4},\infty\right),
$$
the interval for which there is no good profile, while for $K=K^*=(m-1)^2/4$ the only remaining possibility is that the orbit going out of $P_0$ enters the unified critical point $Q_1=Q_4$ (which is a saddle-node). Thus, the case $K=(m-1)^2/4$ should be added to the interval of existence and uniqueness of the good profile with interface. Finally, recalling that $(m-p)/(m-1)=2$ in our case and letting $K=(m-1)^2/4=K^*$ in \eqref{param}, we obtain
$$
\frac{(m-1)^2}{4}=\frac{4m}{(\alpha^*)^2},
$$
which gives the explicit value of $\alpha^*$ in \eqref{alpha.equal}.
\end{proof}

\section{Analysis in dimensions $N=2$ and $N=1$}\label{sec.lowdim}

In dimensions $N=2$ and $N=1$ there are a few technical differences in the previous analysis, although the results are the same. For the reader's convenience, we decided not to merge them in the previous sections and devote a separate chapter to them. We give below only the differences with respect to the previous sections.

\subsection{Dimension $N=2$}\label{subsec.N2}

In dimension $N=2$, the system \eqref{PPSyst} becomes
\begin{equation}\label{PPSystN2}
\left\{\begin{array}{ll}\dot{X}=X(2-(m-1)Y),\\ \dot{Y}=-mY^2+2X-(m-1)XY-KX^{(m-p)/(m-1)},\end{array}\right.
\end{equation}
and the critical points $P_0$ and $P_1$ coincide at the origin (that we keep calling $P_0$). The local behavior of the orbits in the neighborhood of this critical point is analyzed below.
\begin{lemma}\label{lem.P0P1}
The critical point $P_0=(0,0)$ of the system \eqref{PPSystN2} is a saddle-node. There is a unique orbit going out of $P_0$ into the phase plane tangent to the line $Y=X$, containing profiles with the good behavior: $f(0)=A>0$, $f'(0)=0$. All the other orbits go out of $P_0$ tangent to the $Y$ axis into the region $\{Y<0\}$ and contain profiles such that
\begin{equation}\label{beh.P02}
f(\xi)\sim D\left(-\ln\,\xi\right)^{1/m}, \qquad {\rm as} \ \xi\to0, \ D>0 \ {\rm free \ constant}.
\end{equation}
\end{lemma}
\begin{proof}
The linearization of the system in a neighborhood of the point $P_0$ has the matrix
$$
M(P_0)=\left(
  \begin{array}{cc}
    2 & 0 \\
    2 & 0 \\
  \end{array}
\right),
$$
with eigenvalues $\lambda_1=2$ and $\lambda_2=0$. By introducing the change of variable $Z=Y-X$, the system \eqref{PPSyst} becomes
\begin{equation}\label{PPSyst.N2}
\left\{\begin{array}{ll}\dot{X}=X(2-(m-1)X-(m-1)Z),\\ \dot{Z}=-m(X+Z)^2-KX^{(m-p)/(m-1)},\end{array}\right.
\end{equation}
and we can apply \cite[Theorem 1, Section 2.11]{Pe} to its origin (which is the same critical point as $P_0$) to conclude that it is a saddle-node. By the theory in \cite[Section 3.4]{GH}, there exists only one orbit going out of $P_0$ tangent to the eigenvector $(1,1)$ corresponding to the eigenvalue $\lambda_1=2$, that contains profiles whose analysis is completely similar to the one in Lemma \ref{lem.P0} (for the particular value $N=2$) leading to the good behavior $f(0)=A>0$, $f'(0)=0$. All the other orbits are tangent to the $Y$ axis, as this is the direction of the eigenvector $e_2=(0,1)$ corresponding to the eigenvalue $\lambda_2=0$. By neglecting the lower order terms we obtain that these trajectories satisfy in a neighborhood of $P_0$
$$
\frac{dY}{dX}\sim-\frac{mY^2}{2X}
$$
which after integration leads to
\begin{equation}\label{interm18}
\frac{1}{Y}\sim\frac{m}{2}\ln\,X,
\end{equation}
which shows that all these orbits go into the region $\{Y<0\}$ of the phase plane. Since \eqref{interm18} is difficult to integrate, we obtain the first order approximation as $\xi\to0$ directly from the equation \eqref{ODE} by showing that the last three terms are negligible on the orbits having $X\to0$, $Y\to0$ and $Y/X\to-\infty$ with respect to the previous two. Indeed we have
\begin{equation}\label{interm19}
\frac{\beta\xi f'(\xi)}{(f^m)'(\xi)/\xi}=\frac{\beta}{m}\xi^{2}f(\xi)^{1-m}=\frac{2\beta}{\alpha}X(\xi)\to0,
\end{equation}
then
\begin{equation}\label{interm20}
\frac{\alpha f(\xi)}{(f^m)'(\xi)/\xi}=\frac{\alpha\xi^2}{mf(\xi)^{m-1}Y(\xi)}=\frac{2X(\xi)}{Y(\xi)}\to0
\end{equation}
and finally
\begin{equation}\label{interm21}
\frac{\xi^{\sigma}f(\xi)^p}{(f^m)'(\xi)/\xi}=\frac{\xi^{\sigma+2}f(\xi)^p}{mf(\xi)^mY(\xi)}=C\frac{X^{(m-p)/(m-1)}}{Y}\to0,
\end{equation}
since $(m-p)/(m-1)>1$. Gathering the limits in \eqref{interm19}, \eqref{interm20} and \eqref{interm21}, we find that the first approximation of the local behavior on the orbits going out of $P_0$ tangent to the $Y$ axis is given by the joint effect of the first two terms in \eqref{ODE} (having the same homogeneity), that is
$$
(f^m)''(\xi)+\frac{1}{\xi}(f^m)'(\xi)\sim0,
$$
which leads to \eqref{beh.P02} by integration. The limit is taken as $\xi\to0$, as in the other possible cases $\xi\to\xi_0\in(0,\infty)$ or $\xi\to\infty$ we reach an immediate contradiction with the fact that $X\to0$, we omit these easy details which go in the same way as in the proofs of Lemmas \ref{lem.P0} and \ref{lem.P1}.
\end{proof}
Regarding the global analysis, despite the fact that we are no longer dealing with a saddle-saddle connection, we need to look at the unique orbit going out of $P_0$ tangent to the eigenvector $(1,1)$, which is similar to the one going out of $P_0$ in dimensions $N\geq3$. The analysis for $K$ close to zero brings a difference: the two branches $Y_1(X)$ and $Y_2(X)$ defined in \eqref{branches} obviously intersect at $X=0$ as they start from the same point. The function $\Delta(X)$ becomes
$$
\Delta(X)=X\left[(m-1)^2X+8m-4KmX^{(1-p)/(m-1)}\right]
$$
and for $K$ sufficiently small one can check that the term in brackets above is always positive. Indeed, on the one hand
\begin{equation}\label{interm22}
8m>4KmX^{(1-p)/(m-1)}, \qquad {\rm for} \ 0<X<\left(\frac{2}{K}\right)^{(m-1)/(1-p)},
\end{equation}
while on the other hand
\begin{equation}\label{interm23}
(m-1)^2X>4KmX^{(1-p)/(m-1)}, \qquad {\rm for} \ X^{(m+p-2)/(m-1)}>\frac{4Km}{(m-1)^2},
\end{equation}
and the two intervals of $X$ in which \eqref{interm22} and \eqref{interm23} hold true obviously overlap for $K$ sufficiently small if $m+p\geq2$, proving that for $K>0$ sufficiently small the branches $Y_1(X)$ and $Y_2(X)$ do not intersect again and the regions (I), (II), (III), (IV) illustrated in Figure \ref{fig1} are the same. The rest of the global analysis in both cases $m+p>2$ and $m+p=2$ is perfectly similar for $N=2$ since the term $(N-2)$ is no longer essential in the remaining estimates, as an inspection of the proofs shows.

\subsection{Dimension $N=1$}\label{subsec.N1}

In dimension $N=1$ we have again some differences concerning the local analysis of the critical points $P_0=(0,0)$ and $P_1=(0,1/m)$, more precisely
\begin{lemma}\label{lem.P01}
For $N=1$, the critical point $P_0=(0,0)$ is an unstable node. There is a unique orbit going out of $P_0$ tangent to the direction of the vector $(1,2)$ which contains good profiles such that $f(0)=A>0$, $f'(0)=0$. All the other orbits contain profiles such that $f(0)=A>0$ with any possible slope $f'(0)=B\neq0$. The critical point $P_1=(0,1/m)$ is a saddle point. The only orbit going out of this point contains profiles such that
\begin{equation}\label{interm24}
f(\xi)\sim D\xi^{1/m}, \qquad {\rm as} \ \xi\to0, \ D>0 \ {\rm free \ constant}.
\end{equation}
\end{lemma}
\begin{proof}
The linearization of the system \eqref{PPSyst} in a neighborhood of the critical points $P_0$ and $P_1$ has the matrices
$$
M(P_0)=\left(
         \begin{array}{cc}
           2 & 0 \\
           2 & 1\\
         \end{array}
       \right), \qquad
M(P_1)=\left(
         \begin{array}{cc}
           \frac{m+1}{m} & 0 \\
           \frac{m+1}{m} & -1 \\
         \end{array}
       \right),
$$
thus $P_0$ is an unstable node and $P_1$ is a saddle point. The unique orbit going out of $P_1$ contains profiles such that $Y=1/m$ and $X\to0$, which in terms of profiles gives
$$
\frac{f'(\xi)}{f(\xi)}\sim\frac{1}{m\xi},
$$
leading to the behavior given by \eqref{interm24} after integration. With respect to the point $P_0$, the matrix $M(P_0)$ has eigenvalues $\lambda_1=2$ and $\lambda_2=1$ with corresponding eigenvectors $e_1=(1,2)$ and $e_2=(0,1)$. Since $\lambda_2<\lambda_1$, there exists a unique orbit going out of $P_0$ tangent to the eigenvector $e_1$, and the behavior of the profiles contained in it is given by
$$
\frac{X}{Y}\sim\frac{1}{2},
$$
that gives good profiles with local behavior given by \eqref{interm1} as $\xi\to0$, as it is obtained by following line to line the end of the proof of Lemma \ref{lem.P0}. All the other orbits going out of $P_0$ are tangent to the eigenvector $e_2$, thus $Y/X\to\pm\infty$. We infer then by inspection of the system \eqref{PPSyst} and keeping for the first order approximation only the dominating terms that
$$
\dot{X}\sim2X, \qquad \dot{Y}\sim Y,
$$
hence by integration $Y\sim CX^{1/2}$, which in terms of profiles becomes
$$
f^{(m-3)/2}(\xi)f'(\xi)\sim D
$$
or equivalently
\begin{equation}\label{interm25}
f(\xi)\sim\left(D_1+D\xi\right)^{2/(m-1)}, \qquad D_1>0, \ D\neq0 \ {\rm free \ constants}
\end{equation}
and one can check that this behavior holds true as $\xi\to0$ by standard arguments already employed in the proof of Lemmas \ref{lem.P0} and \ref{lem.P1}. We conclude from \eqref{interm25} that such profiles satisfy $f(0)>0$ and $f'(0)$ can take any non-zero value, as claimed.
\end{proof}
This analysis shows that we have to follow again the unique orbit going out of $P_0$ in direction of the eigenvector $e_1=(1,2)$. The rest of the analysis is practically similar as the one performed in dimension $N\geq3$ and we omit the details here.

\section{Non-existence in the case $m+p<2$}\label{sec.lower}

The analysis of the remaining range $m+p<2$ goes with similar arguments as in the previous sections and we will be rather brief in order to avoid extending too much the current work. Since the local analysis of the critical points in the plane $P_0$ and $P_1$ is completely similar as in Lemmas \ref{lem.P0} and \ref{lem.P1} (or their analogous Lemmas \ref{lem.P0P1} and \ref{lem.P01} in dimensions $N=2$, respectively $N=1$), thus the differences appear again when analyzing the critical points at infinity. We notice that $(m-p)/(m-1)>2$ in this case, thus the term
$$
-K\overline{X}^{(2m-p-1)/(m-1)}W^{(m+p-2)/(m-1)}
$$
becomes the dominating negative power of $W$ in \eqref{interm0}. Thus, following the theory in \cite[Section 3.10]{Pe}, we have to take as factor the inverse of the highest power of $W$, which is now $W^{(m-p)/(m-1)}$, letting now
$$
P^*(\overline{X},\overline{Y},W)=W^{(m-p)/(m-1)}P\left(\frac{\overline{X}}{W},\frac{\overline{Y}}{W}\right), \ \  Q^*(\overline{X},\overline{Y},W)=W^{(m-p)/(m-1)}Q\left(\frac{\overline{X}}{W},\frac{\overline{Y}}{W}\right)
$$
and get instead of \eqref{interm0} the following equation
\begin{equation*}\label{interm0bis}
\begin{split}
\overline{X}Q^*(\overline{X},\overline{Y},W)&-\overline{Y}P^*(\overline{X},\overline{Y},W)=-m\overline{X}\overline{Y}^2W^{(m-p)/(m-1)-2}\\
&-(N-2)\overline{X}\overline{Y}W^{(m-p)/(m-1)-1}+2\overline{X^2}W^{(m-p)/(m-1)-1}\\
&-(m-1)\overline{X}^2\overline{Y}W^{(m-p)/(m-1)-2}-K\overline{X}^{(2m-p-1)/(m-1)}\\
&-2\overline{X}\overline{Y}W^{(m-p)/(m-1)-1}+(m-1)\overline{X}\overline{Y}^2W^{(m-p)/(m-1)-2}=0,
\end{split}
\end{equation*}
and we infer from letting $W=0$ in the above equation that
$$
K\overline{X}^{(2m-p-1)/(m-1)}=0,
$$
thus the only critical points are $Q_2=(0,1,0)$ and $Q_3=(0,-1,0)$ (in variables $(\overline{X},\overline{Y},W)$). It thus follows that the orbit going out of $P_0$ must enter the critical point $Q_3$, which is characterized by the fact that $X\to\infty$, $Y\to-\infty$ and $Y/X\to-\infty$ on the orbits approaching it. A direct local analysis of the critical point is very involved due to the fact that it requires further changes of variable, but at a formal level, we find by keeping the dominating terms in the two equations of the system \eqref{PPSyst} that
$$
\frac{dY}{dX}\sim\frac{mY^2+KX^{(m-p)/(m-1)}}{(m-1)XY},
$$
which after integration reads
\begin{equation}\label{interm27}
Y^2+\frac{2K}{m+p}X^{(m-p)/(m-1)}-CX^{2m/(m-1)}=0, \qquad C\in\real \ {\rm free \ constant}
\end{equation}
Noticing that $2m/(m-1)>(m-p)/(m-1)$ we readily infer from \eqref{interm27} that $Y\sim CX^{m/(m-1)}$ as $X\to\infty$, with $C<0$ (the case $C>0$ corresponding to orbits going out of $Q_2$ where $Y/X\to+\infty$). We further get by passing to profiles and undoing the change of variable \eqref{change} that
$$
f^{m-1}(\xi)f'(\xi)\sim C\xi^{(m+1)/(m-1)},
$$
which leads to the same behavior as in \eqref{interm6} with the subsequent analysis at the end of the proof of Lemma \ref{lem.Q23} giving that such behavior is taken as $\xi\to\xi_0\in(0,\infty)$. We thus infer that no interface behavior is possible in the range $m+p<2$, and moreover, all the good profiles have a change of sign at some finite positive point. This concludes the proof of Theorem \ref{th.2}. A fully rigorous proof can be done by changing the phase plane system in the line with \cite[Section 7]{IS20b}.

\section{Transformations to other reaction-convection-diffusion equations}\label{sec.transf}

In this final section, we introduce a transformation that is, up to our knowledge, new, mapping radially symmetric solutions to our Eq. \eqref{eq1} into solutions to reaction-convection-diffusion equations, allowing thus for the classification of the special solutions in form of \emph{traveling waves} for the latter. More precisely, let us begin with the radially symmetric form of Eq. \eqref{eq1}, that is
\begin{equation}\label{eq1.radial}
u_t=(u^m)_{rr}+\frac{N-1}{r}u_r+r^\sigma u^p.
\end{equation}
We next introduce the general change of variable
\begin{equation}\label{interm26}
u(r,t)=s^{\gamma/\theta}w(s,\tau), \qquad s=r^\theta, \tau =\theta^2 t,
\end{equation}
where $\gamma$ and $\theta$ will be chosen later. We obtain from \eqref{eq1.radial} the partial differential equation satisfied by the new function $w$
\begin{eqnarray}\label{eq2}
w_\tau&=&s^{\frac{(m-1)\gamma-2}{\theta}+2}(w^m)_{ss}+\frac{2m\gamma+N+\theta-2}\theta s^{\frac{(m-1)\gamma-2}{\theta}+1}(w^m)_s+\nonumber\\
&+&\frac{m\gamma(m\gamma+N-2)}{\theta^2}s^{\frac{(m-1)\gamma-2}{\theta}}w^m+\frac{s^{\frac{(p-1)\gamma+\sigma}{\theta}}}{\theta^2}w^p
\end{eqnarray}

%
	
\medskip

\noindent \textbf{The transformation}. An interesting application of the previous generic change of variable consists in regarding Eq. \eqref{eq2} as an equation of Euler type. To this end, we let
$$
\gamma=\frac{2}{m-1}, \qquad \theta=1,
$$
with $\sigma=2(1-p)/(m-1)$, as we consider throughout the whole paper. By further letting $s=e^y$ and performing straightforward calculations, we map Eq. \eqref{eq1} into the following reaction-convection-diffusion equation in one dimension
\begin{equation}\label{eq4}
w_\tau=(w^m)_{yy}+\frac{N(m-1)+2(m+1)}{(m-1)}(w^m)_y+\frac{2m(N(m-1)+2)}{(m-1)^2}w^m+w^p,
\end{equation}
where $N>0$ becomes in fact a parameter of the equation and can be chosen as a real number in order to get more general coefficients. It is a well-known fact that equations of the form of Eq. \eqref{eq4} have usually special solutions in the form of \emph{traveling waves}, which are also a class of eternal solutions. More precisely, such solutions have the general form $w(y,\tau)=f(y+c\tau)$, where $f$ is the profile of the wave and $c$ is the speed of advance of it. We next show that the traveling waves of Eq. \eqref{eq4} are mapped through the transformation \eqref{interm26} into self-similar solutions of exponential type to Eq. \eqref{eq1}. Starting from a generic traveling wave and undoing the change of variables we get
$$
w(s,\tau)=f(\ln(s)+c\tau)=f(\ln(s)+\ln(e^{c\tau}))=f(\ln(se^{c\tau}))
$$
whence
$$
u(r,t)=r^\gamma f(\ln(re^{ct}))=e^{-c\gamma t}(re^{ct})^\gamma f(\ln(re^{ct}))=e^{-\alpha t}f_1(re^{\beta t})
$$
with $\beta=c$ and $\alpha=c\gamma=2\beta/(m-1)$, obtaining thus an \emph{eternal solution} to Eq. \eqref{eq1}. Theorems \ref{th.1}, \ref{th.1bis} and \ref{th.2} can be straightforwardly mapped into results stating either the existence and uniqueness of traveling waves with minimal speed $c=c^*$ for Eq. \eqref{eq4} or non-existence of any traveling wave. We gather them in a single statement below.
\begin{theorem}\label{th.TW}
If $m+p\geq2$, then there exists a unique $c^*>0$ such that for any $c\geq c^*$, there exists a unique (up to translations) traveling wave with speed $c$. All these traveling waves are compactly supported to the right, and if $m+p>2$ the unique traveling wave with minimal speed $c=c^*$ has an interface behavior as in \eqref{TypeI}, while all the traveling waves with speed $c>c^*$ have interface behavior as in \eqref{TypeII}. Finally, if $m+p<2$ there are no traveling waves to Eq. \eqref{eq4}.
\end{theorem}
In general, equations of the form \eqref{eq1} with all possible combinations of $m$, $p$ and $\sigma$ have a bunch of possible and interesting transformations into equations either of the same form (self-maps) or different. A thorough study of such mappings and their applications in the form of new solutions and behaviors will be addressed in a forthcoming work.

\bigskip

\noindent \textbf{Acknowledgements} A. S. is partially supported by the Spanish project MTM2017-87596-P.

\bibliographystyle{plain}

\end{document}